\documentclass[11pt,a4paper]{amsart}
\usepackage{amsfonts,amsmath,amssymb,amsthm,mathtools}
\usepackage[noadjust]{cite}
\usepackage[textsize=tiny]{todonotes}

\numberwithin{equation}{section}

\DeclareMathOperator{\Ker}{Ker}
\DeclareMathOperator{\Ran}{Ran}
\DeclareMathOperator{\Dom}{Dom}
\DeclareMathOperator{\spec}{spec}
\DeclareMathOperator{\sign}{sign}
\DeclareMathOperator{\Tan}{Tan}
\DeclareMathOperator{\divgc}{div}
\DeclareMathOperator{\grad}{grad}

\DeclareSymbolFont{SY}{U}{psy}{m}{n}
\DeclareMathSymbol{\emptyset}{\mathord}{SY}{'306}

\DeclarePairedDelimiter{\abs}{|}{|}
\DeclarePairedDelimiter{\norm}{\lVert}{\rVert}
\DeclarePairedDelimiter{\scprod}{\langle}{\rangle}

\newcommand{\dd}{\mathrm d}

\newcommand{\eps}{\varepsilon}

\newcommand{\ess}{\mathrm{ess}}

\newcommand{\NN}{\mathbb{N}}
\newcommand{\RR}{\mathbb{R}}
\newcommand{\CC}{\mathbb{C}}
\newcommand{\EE}{\mathsf{E}}

\newcommand{\cC}{{\mathcal C}}
\newcommand{\cD}{{\mathcal D}}
\newcommand{\cE}{{\mathcal E}}
\newcommand{\cH}{{\mathcal H}}
\newcommand{\cK}{{\mathcal K}}

\newcommand{\fa}{{\mathfrak a}}
\newcommand{\fb}{{\mathfrak b}}

\newcommand{\fv}{{\mathfrak v}}
\newcommand{\fD}{{\mathfrak D}}
\newcommand{\fM}{{\mathfrak M}}
\newcommand{\fN}{{\mathfrak N}}

\theoremstyle{plain}
\newtheorem{theorem}{Theorem}[section]
\newtheorem{proposition}[theorem]{Proposition}
\newtheorem{lemma}[theorem]{Lemma}
\newtheorem{corollary}[theorem]{Corollary}

\theoremstyle{definition}
\newtheorem{hypothesis}[theorem]{Hypothesis}

\theoremstyle{remark}
\newtheorem{remark}[theorem]{Remark}

\title[On a minimax principle in spectral gaps]{On a minimax principle in spectral gaps}
\subjclass[2010]{Primary 49Rxx; Secondary 47A10, 47A75}
\keywords{Minimax values, eigenvalues in gap of the essential spectrum, block diagonalization, Stokes operator}
\date{}

\author[A.\ Seelmann]{Albrecht Seelmann}
\address{A.~Seelmann,
 Technische Univer\-si\-t\"at Dortmund, Fakult\"at f\"ur Mathematik, D-44221 Dortmund, Germany}
\email{albrecht.seelmann@math.tu-dortmund.de}

\begin{document}

\begin{abstract}
  The minimax principle for eigenvalues in gaps of the essential spectrum in the form presented by Griesemer, Lewis, and
  Siedentop in~[Doc.~Math.~\textbf{4} (1999), 275--283] is adapted to cover certain abstract perturbative settings with bounded
  or unbounded perturbations, in particular ones that are off-diagonal with respect to the spectral gap under consideration. This
  in part builds upon and extends the considerations in the author's appendix
  to~[J.~Spectr.~Theory~\textbf{10} (2020), 843--885]. Several monotonicity and continuity properties of eigenvalues in gaps of
  the essential spectrum are deduced, and the Stokes operator is revisited as an example.
\end{abstract}

\maketitle

\section{Introduction and main result}\label{sec:intro}

The standard Courant minimax values $\lambda_k(A)$ of a lower semibounded operator $A$ on a Hilbert space $\cH$ are given by
\begin{equation*}
  \lambda_k(A)
  =
  \inf_{\substack{\fM\subset\Dom(A)\\ \dim\fM=k}} \sup_{\substack{x\in\fM\\ \norm{x}=1}} \scprod{ x , Ax }
  =
  \inf_{\substack{\fM\subset\Dom(|A|^{1/2})\\ \dim\fM=k}} \sup_{\substack{x\in\fM\\ \norm{x}=1}} \fa[x,x]
\end{equation*}
for $k\in\NN$ with $k\le\dim\cH$, see, e.g.,~\cite[Theorem~12.1]{LL01} and also~\cite[Section~12.1 and Exercise~12.4.2]{Schm12}.
Here, $\langle\cdot,\cdot\rangle$ denotes the inner product of $\cH$, and $\fa$ with
$\fa[x,x] = \scprod{ \abs{A}^{1/2}x , \sign(A)\abs{A}^{1/2}x }$ for $x\in\Dom(\abs{A}^{1/2})$ is the form associated with $A$.

The above minimax values have proved to be a powerful description of the eigenvalues below the essential spectrum of $A$; in
fact, they agree with these eigenvalues in nondecreasing order counting multiplicities as long as they exists and else equal the
bottom of the essential spectrum. A standard application in this context is that the eigenvalues below the essential spectrum
exhibit a monotonicity with respect to the operator: for two lower semibounded self-adjoint operators $A$ and $B$ with $A\le B$
in the sense of quadratic forms one has $\lambda_k(A) \le \lambda_k(B)$ for all $k$, see, e.g.,~\cite[Corollary~12.3]{Schm12}.

Matters get, however, much more complicated when eigenvalues in a gap of the essential spectrum are considered. If $A_+$ is the
(lower semibounded) part of $A$ associated with its spectrum in an interval of the form $(\gamma,\infty)$, $\gamma\in\RR$, then
the minimax values for $A_+$ still describe the eigenvalues of $A_+$ below its essential spectrum, and thus the eigenvalues of
$A$ in $(\gamma, \infty)$ below the essential spectrum of $A$ above $\gamma$. However, the subspaces over which the corresponding
infimum is taken are chosen within the spectral subspace for $A$ associated with the interval $(\gamma,\infty)$ and therefore
usually depend on the operator itself rather than just its domain. This makes it difficult to compare minimax values in spectral
gaps of two different operators $A$ and $B$, even if their domains agree.

An adapted minimax principle taking this problem into account was first proposed by Talman~\cite{Tal86} and Datta and
Devaiah~\cite{DD88} in the context of Dirac operators. A corresponding mathematically rigorous result was announced by
Esteban and S\'er\'e in~\cite{ES97} and proved together with Dolbeault in~\cite{DES00a}.
To the best of the authors knowledge, the first abstract theorem in this direction is due to Griesemer and Siedentop~\cite{GS99},
the hypotheses of which have an overlap with~\cite{ES97,DES00a} but do not seem suitable to handle Dirac operators efficiently.
In an attempt to overcome this and as a step towards finding the optimal assumptions, Griesemer, Lewis, and
Siedentop~\cite{GLS99} provided an alternative set of hypotheses.
In a parallel development, Dolbeault, Esteban, and S\'er\'e~\cite{DES00b} obtained an abstract theorem in an operator setting
with yet another set of hypotheses that has an overlap with those of~\cite{GLS99} but allows to deal with more potentials for
Dirac operators. However, the abstract result of~\cite{GLS99} does not seem to be contained in~\cite{DES00b}. The result
in~\cite{DES00b} has later been extended to a form setting by Morozov and M\"uller~\cite{MM15}, and recently by Schimmer,
Solovej, and Tokus~\cite{SST20} to a class of symmetric operators with a distinguished self-adjoint extension.
There has also been some activity regarding variational principles for block operator matrices, see, e.g.,~\cite{LLT02,KLT04} and
triple variational principles, see, e.g.,~\cite{EL02,LS16}. They, however, follow a different approach and are not pursued here.

The aim of the present work is to complement the above works in an abstract framework. To this end, the result by Griesemer,
Lewis, and Siedentop in ~\cite{GLS99} is adapted to a perturbative setting. In the particular case of bounded additive
perturbations, this has already been done by the present author in the appendix to~\cite{NSTTV20} with hypotheses that can, under
reasonable assumptions, be verified explicitly by means of the Davis-Kahan $\sin2\Theta$ theorem from~\cite{DK70} or variants
thereof. The latter has been successfully applied in~\cite{NSTTV20} to study lower bounds on the movement of eigenvalues in gaps
of the essential spectrum and of edges of the essential spectrum. In the current work, the considerations
from~\cite[Appendix~A]{NSTTV20} are extended and supplemented to cover also certain unbounded perturbations, in particular ones
that are off-diagonal with respect to the spectral gap under consideration. The results obtained here seem~\emph{not} suitable to
handle Dirac operators with Coulomb potentials since either the perturbation is assumed to be sufficiently small
(Theorem~\ref{thm:genOpInfinitesimal}) or of an off-diagonal structure (Theorem~\ref{thm:offdiagOp}) or since they assume a
semibounded setting (Theorems~\ref{thm:genSemibounded} and~\ref{thm:offdiagForm}). However, weaker perturbations and other
important situations such as perturbed periodic Schr\"odinger operators seem to be a natural context in which they can be
applied. It should also be mentioned that some of the results and applications discussed here might at least in parts also be
obtained with the approaches from earlier works such as~\cite{DES00b,MM15,SST20,LS16}. This is commented on at various spots
below, see, e.g., Remarks~\ref{rem:offdiagForm},~\ref{rem:boundedPert}\,(2),~\ref{rem:StokesMMSST}, and~\ref{rem:gap}. The
present work focuses on~\cite{GLS99} as a starting point for mainly two reasons: Firstly, the proof of that result is remarkably
elementary and short, while the proofs of~\cite{DES00b, MM15, SST20, LS16} are each a lot longer and much more technical, and,
secondly, the techniques employed in Sections~\ref{sec:graphNorm} and~\ref{sec:blockDiag} below to apply the approach
from~\cite{GLS99} promise to be of independent interest. In any case, to the best of the authors knowledge, neither the main
results presented here nor their applications have been stated explicitly anywhere before. Rare exceptions to the latter are
commented on accordingly.

\subsection*{Main results}
In order to formulate our main results, it is convenient to fix the following notational setup tailored towards spectral gaps to
the right of $0$; other gaps can of course always be reduced to this situation by spectral shift,
cf.~Remark~\ref{rem:spectralShift} below.

\begin{hypothesis}\label{hyp:minimax}
  Let $A$ be a self-adjoint operator on a Hilbert space. Denote the spectral projections for $A$ associated with the intervals
  $(0,\infty)$ and $(-\infty,0]$ by $P_+$ and $P_-$, respectively, that is,
  \begin{equation*}
    P_+ := \EE_A\bigl((0,\infty)\bigr)
    ,\quad
    P_- := I-P_+
    ,
  \end{equation*}
  and let
  \begin{equation*}
    \cD_\pm := \Ran P_\pm \cap \Dom(A)
    ,\quad
    \fD_\pm := \Ran P_\pm \cap \Dom(\abs{A}^{1/2})
    .
  \end{equation*}
  Moreover, let $B$ be another self-adjoint operator on the same Hilbert space with analogously defined spectral projections
  \begin{equation*}
    Q_+ := \EE_B\bigl((0,\infty)\bigr)
    ,\quad
    Q_- := I - Q_+
    ,
  \end{equation*}
  and denote by $\fb$ the form associated with $B$, that is,
  \begin{equation*}
    \fb[x,y] = \scprod{ \abs{B}^{1/2}x , \sign(B)\abs{B}^{1/2}y }
  \end{equation*}
  for $x,y\in\Dom[\fb] = \Dom(\abs{B}^{1/2})$.
\end{hypothesis}

Here, $\EE_A$ and $\EE_B$ stand for the projection-valued spectral measures for the operators $A$ and $B$, respectively, and
$\Ran P_\pm$ denotes the range of $P_\pm$. We have also used the notation $I$ for the identity operator.

Denoting the form associated with $A$ by $\fa$, the minimax values of the positive part $A|_{\Ran P_+}$ of $A$ can clearly be
written as
\begin{equation*}
  \lambda_k(A|_{\Ran P_+})
  =
  \inf_{\substack{\fM_+\subset\cD_+\\ \dim\fM_+=k}} \sup_{\substack{x\in\fM_+ \oplus \cD_-\\ \norm{x}=1}} \scprod{ x , Ax }
  =
  \inf_{\substack{\fM_+\subset\fD_+\\ \dim\fM_+=k}} \sup_{\substack{x\in\fM_+ \oplus \fD_-\\ \norm{x}=1}} \fa[x,x]
\end{equation*}
for $k\in\NN$ with $k \le \dim \Ran P_+$. The point of interest is now to find conditions on $B$ under which the minimax values
for the positive part $B|_{\Ran Q_+}$ of $B$ admit the same representations with $\scprod{x , Ax }$ and $\fa[x , x]$ replaced by
$\scprod{x , Bx}$ and $\fb[x , x]$, respectively, but with the infima taken over the same respective families of subspaces as for
$A$ above. It is natural to consider this in a perturbative framework where $B$ is obtained by an operator or form perturbation
of $A$ and, thus, one has $\Dom(A) = \Dom(B)$ and/or $\Dom(\abs{A}^{1/2}) = \Dom(\abs{B}^{1/2})$.

In the situation of Hypothesis~\ref{hyp:minimax}, a representation for the minimax values of $B|_{\Ran Q_+}$ of the above
mentioned form is guaranteed by~\cite[Theorem~1]{MM15} in the form setting with $\Dom(\abs{A}^{1/2}) = \Dom(\abs{B}^{1/2})$ if
\begin{equation}\label{eq:MM}
  \sup_{x_- \in \fD_-} \fb[ x_- , x_- ]
  \leq
  0
  <
  \inf_{x_+ \in \fD_+\setminus\{0\}} \sup_{x_- \in \fD_-} \frac{\fb[ x_+ + x_- , x_+ + x_- ]}{\norm{x_+ + x_-}^2}
  ,
\end{equation}
or by~\cite[Theorem~1.1]{DES00b} in the operator setting with $\Dom(A) = \Dom(B)$ if the analogous condition with $\fD_\pm$
replaced by $\cD_\pm$ is satisfied; as pointed out in~\cite{SST20}, for the latter additionally the restriction
$P_-B|_{\Ran P_-}$ should be essentially self-adjoint on $\cD_-$. In case of~\eqref{eq:MM}, the right-hand side of~\eqref{eq:MM}
then agrees with $\lambda_1(B|_{\Ran Q_+})$, so that the strict inequality in~\eqref{eq:MM} is also a necessary condition for
such a representation to hold if $B$ has a spectral gap to the right of zero. However, this strict inequality is not always very
convenient to verify or it is sometimes not even entirely clear how to verify it, cf.~Remarks~\ref{rem:offdiagForm}\,(2)
and~\ref{rem:boundedPert}\,(2) below.

Instead of~\eqref{eq:MM},~\cite{GLS99} used the conditions
\begin{equation}\label{eq:GLS}
  \sup_{x_- \in \fD_-} \fb[ x_- , x_- ]
  \leq
  0
  \quad\text{ and }\quad
  \norm{ (\abs{A}+I)^{1/2} P_+Q_- (\abs{A}+I)^{-1/2} }
  <
  1
  ,
\end{equation}
cf.~Remark~\ref{rem:neg} below, which the authors were able to handle in case of Dirac operators but where especially the second
condition seems to be hard to deal with in a general abstract setting. However, although~\eqref{eq:MM} can treat more
Coulomb-like potentials than~\eqref{eq:GLS} in case of the Dirac operator,~\eqref{eq:GLS} does not seem to imply~\eqref{eq:MM}
directly.

In the main results below the aim is to discuss situations where the second condition in~\eqref{eq:GLS} can be replaced by
$\norm{ P_+Q_- } < 1$, $\norm{ P_+ - Q_+ } < 1$, or by a certain explicit structural assumption on how $B$ is related to $A$.
Here, especially the first two conditions seem to be natural since they relate the subspaces $\Ran P_+$ and $\Ran Q_+$.
Four results in this direction are presented here, each addressing a different situation, which are not contained in the
previously known results in the sense that their hypotheses do not seem to imply~\eqref{eq:MM}, its operator analogue,
or~\eqref{eq:GLS} directly. We first treat the case of operator perturbations and start with the direct extension
of~\cite[Theorem~A.2]{NSTTV20} to infinitesimal perturbations. Recall that an operator $V$ with $\Dom(V) \supset \Dom(A)$ is
called $A$-bounded with $A$-bound $b_* \ge 0$ if for all $b > b_*$ there is some $a \ge 0$ with
\begin{equation*}
  \norm{ Vx }
  \le
  a\norm{x} + b\norm{Ax}
  \quad\text{ for all }\
  x \in \Dom(A)
\end{equation*}
and if there is no such $a$ for $0 < b < b_*$. If $b_* = 0$, then $V$ is called infinitesimal with respect to $A$.

\begin{theorem}\label{thm:genOpInfinitesimal}
  Assume Hypothesis~\ref{hyp:minimax}. Suppose, in addition, that $B$ is of the form $B = A + V$, $\Dom(B) = \Dom(A)$, with some
  symmetric operator $V$ that is infinitesimal with respect to $A$. Furthermore, suppose that we have $\norm{P_+Q_-} < 1$ and
  that
  \begin{equation*}
    \scprod{ x , Bx } \le 0 \quad\text{ for all }\ x \in \cD_-.
  \end{equation*}
  Then,
  \begin{equation*}
    \lambda_k(B|_{\Ran Q_+})
    =
    \inf_{\substack{\fM_+\subset\cD_+\\ \dim\fM_+=k}} \sup_{\substack{x\in\fM_+ \oplus \cD_-\\ \norm{x}=1}} \scprod{ x , Bx }
    =
    \inf_{\substack{\fM_+\subset\fD_+\\ \dim\fM_+=k}} \sup_{\substack{x\in\fM_+ \oplus \fD_-\\ \norm{x}=1}} \fb[x,x]
  \end{equation*}
  for all $k \in \NN$ with $k \le \dim\Ran P_+$.
\end{theorem}

It is worth to note that every operator of the form $B = A + V$ as in Theorem~\ref{thm:genOpInfinitesimal} is automatically
self-adjoint on $\Dom(B) = \Dom(A)$ by the well-known Kato-Rellich theorem. Two more remarks regarding
Theorem~\ref{thm:genOpInfinitesimal} are in order:
(1)
also certain perturbations $V$ that are not infinitesimal with respect to $A$ can be considered here, but at the cost of a
stronger assumption on $\norm{P_+Q_-}$, see Remark~\ref{rem:relBoundK} below;
(2)
the condition $\norm{P_+Q_-} < 1$ is satisfied if the stronger inequality $\norm{P_+-Q_+} < 1$ holds. In the latter case, the
subspaces $\Ran P_+$ and $\Ran Q_+$ automatically have the same dimension, that is, $\dim \Ran P_+ = \dim \Ran Q_+$, see
Remark~\ref{rem:PQbij}\,(a) below.

The stronger condition~$\norm{P_+-Q_+} < 1$ just mentioned in fact also opens the way to employ a different approach than the one
used to prove Theorem~\ref{thm:genOpInfinitesimal}. This alternative approach has previously been used in the context of block
diagonalization of operators and forms, see Section~\ref{sec:blockDiag} below, and is particularly attractive if the unperturbed
operator $A$ is semibounded.

\begin{theorem}\label{thm:genSemibounded}
  Assume Hypothesis~\ref{hyp:minimax}. Suppose, in addition, that $A$ is semibounded and that $\norm{P_+ - Q_+} < 1$.

  If $\Dom(\abs{A}^{1/2}) = \Dom(\abs{B}^{1/2})$ and $\fb[ x , x ] \le 0$ for all $x \in \fD_-$, then
  \begin{equation}\label{eq:genSemibounded:form}
    \lambda_k(B|_{\Ran Q_+})
    =
    \inf_{\substack{\fM_+\subset\fD_+\\ \dim\fM_+=k}} \sup_{\substack{x\in\fM_+ \oplus \fD_-\\ \norm{x}=1}} \fb[x,x]
  \end{equation}
  for all $k \le \dim\Ran P_+ = \dim\Ran Q_+$. If even $\Dom(A) = \Dom(B)$ and $\scprod{ x , Bx } \le 0$ for all $x \in \cD_-$,
  then also
  \begin{equation}\label{eq:genSemibounded:op}
    \lambda_k(B|_{\Ran Q_+})
    =
    \inf_{\substack{\fM_+\subset\cD_+\\ \dim\fM_+=k}} \sup_{\substack{x\in\fM_+ \oplus \cD_-\\ \norm{x}=1}} \scprod{ x , Bx }
  \end{equation}
  for all $k \le \dim\Ran P_+ = \dim\Ran Q_+$.
\end{theorem}

It should be emphasized that the conditions $\Dom(A) = \Dom(B)$ and $\scprod{ x , Bx } \le 0$ for all $x \in \cD_-$ in
Theorem~\ref{thm:genSemibounded} indeed imply that one has also $\Dom(\abs{A}^{1/2}) = \Dom(\abs{B}^{1/2})$ and
$\fb[ x , x ] \le 0$ for all $x \in \fD_-$, see Lemma~\ref{lem:GLS} below. Note also that in contrast to
Theorem~\ref{thm:genOpInfinitesimal}, Theorem~\ref{thm:genSemibounded} makes no assumptions on how the operator $B$ is related to
$A$. The latter will, however, be relevant when the hypotheses of Theorem~\ref{thm:genSemibounded} are to be verified in concrete
situations.

The condition $\scprod{ x , Bx } \le 0$ for all $x \in \cD_-$ plays an important role in both
Theorems~\ref{thm:genOpInfinitesimal} and~\ref{thm:genSemibounded}. In the case where $B = A + V$ with some $A$-bounded symmetric
operator $V$, this condition is automatically satisfied if $\scprod{ x , Vx } \le 0$ for all $x \in \cD_-$ since
$\scprod{ x , Ax } \le 0$ holds for all $x \in \cD_-$ by definition. The latter is certainly the case for nonpositive $V$.
Another instance of perturbations satisfying $\scprod{ x , Vx } \le 0$ for all $x \in \cD_-$ are so-called~\emph{off-diagonal}
perturbations with respect to the decomposition $\Ran P_+ \oplus \Ran P_-$, in which case also the condition $\norm{P_+-Q_+} < 1$
can be verified efficiently. In comparison with Theorem~\ref{thm:genOpInfinitesimal}, we may even relax the assumption on the
$A$-bound of $V$ here.

\begin{theorem}\label{thm:offdiagOp}
  Assume Hypothesis~\ref{hyp:minimax}. Suppose, in addition, that $B$ has the form $B = A + V$, $\Dom(B) = \Dom(A)$, with some
  symmetric $A$-bounded operator $V$ with $A$-bound smaller than $1$ and which is off-diagonal on $\Dom(A)$ with respect to the
  decomposition $\Ran P_+ \oplus \Ran P_-$, that is,
  \begin{equation*}
    P_+VP_+ x = 0 = P_-VP_- x \quad\text{ for all }\ x \in \Dom(A).
  \end{equation*}
  Then, one has $\dim\Ran P_+ = \dim\Ran Q_+$ and
  \begin{equation*}
    \lambda_k(B|_{\Ran Q_+})
    =
    \inf_{\substack{\fM_+\subset\cD_+\\ \dim\fM_+=k}} \sup_{\substack{x\in\fM_+ \oplus \cD_-\\ \norm{x}=1}} \scprod{ x , Bx }
    =
    \inf_{\substack{\fM_+\subset\fD_+\\ \dim\fM_+=k}} \sup_{\substack{x\in\fM_+ \oplus \fD_-\\ \norm{x}=1}} \fb[x,x]
  \end{equation*}
  for all $k \in \NN$ with $k \le \dim\Ran Q_+$.
\end{theorem}

It is again worth to note that every operator of the form $B = A + V$ as in Theorem~\ref{thm:offdiagOp} is automatically
self-adjoint on $\Dom(B) = \Dom(A)$ by the Kato-Rellich theorem. Moreover, although off-diagonal perturbations may seem a bit
restrictive, they appear quite naturally when a general, not necessarily off-diagonal, perturbation is decomposed into its
diagonal and off-diagonal parts. How Theorem~\ref{thm:offdiagOp} may then be applied is demonstrated in
Proposition~\ref{prop:redOffDiag} below and the considerations thereafter.

The method of proof for Theorem~\ref{thm:offdiagOp} can to some extend be carried over to off-diagonal form perturbations, at
least in the semibounded setting. The latter restriction is commented on in Section~\ref{sec:blockDiag} below.

\begin{theorem}\label{thm:offdiagForm}
  Assume Hypothesis~\ref{hyp:minimax}. Suppose, in addition, that $B$ is semibounded and that its form $\fb$ is given by
  $\fb = \fa + \fv$, $\Dom[\fb] = \Dom[\fa]$, where $\fa$ is the form associated with $A$ and $\fv$ is a symmetric sesquilinear
  form satisfying
  \begin{equation*}
    \fv[ P_+x , P_+y ]
    =
    0
    =
    \fv[ P_-x , P_-y ]
    \quad\text{ for all }\
    x,y \in \Dom[\fa] \subset \Dom[\fv]
  \end{equation*}
  and
  \begin{equation}\label{eq:formRelBound}
    \abs{ \fv[ x , x ] }
    \le
    a\norm{x}^2 + b\abs{\fa[ x , x ]}
    \quad\text{ for all }\
    x \in \Dom(\abs{A}^{1/2}) = \Dom[\fa]
  \end{equation}
  with some constants $a,b \ge 0$.

  Then, one has $\dim\Ran P_+ = \dim\Ran Q_+$ and
  \begin{equation*}
    \lambda_k(B|_{\Ran Q_+})
    =
    \inf_{\substack{\fM_+\subset\fD_+\\ \dim\fM_+=k}} \sup_{\substack{x\in\fM_+ \oplus \fD_-\\ \norm{x}=1}} \fb[x,x]
  \end{equation*}
  for all $k \in \NN$ with $k \le \dim\Ran Q_+$.
\end{theorem}

The semiboundedness of $B$ in Theorem~\ref{thm:offdiagForm} forces $A$ to be semibounded as well, see the proof of
Theorem~\ref{thm:offdiagForm} below. In this regard, Theorem~\ref{thm:offdiagForm} can be interpreted as a particular case of
the first part of Theorem~\ref{thm:genSemibounded} with $\Dom(\abs{A}^{1/2}) = \Dom(\abs{B}^{1/2})$, in which the remaining
hypotheses are automatically satisfied due to the structure of the perturbation.

\begin{remark}\label{rem:offdiagForm}
  (1)
  If $B$ in Theorem~\ref{thm:offdiagForm} is lower semibounded, then the operator $(\abs{B} + I)^{1/2}Q_-$ is everywhere defined
  and bounded, and so is the operator $(\abs{B} + I)^{1/2}Q_-P_+$. Taking into account that $\fb[ x , x ] = \fa[ x , x ] \le 0$
  for $x \in \fD_-$ and $\fb[ x , x ] = \fa[ x , x ] > 0$ for $x \in \fD_+$, Theorem~\ref{thm:offdiagForm} therefore reproduces
  in this situation a particular case of the earlier result~\cite[Theorem~3]{GS99}.

  (2)
  If in Theorems~\ref{thm:offdiagOp} or~\ref{thm:offdiagForm} the~\emph{unperturbed} operator $A$ has a spectral gap to the right
  of $0$, then considerations as in part~(1) show that~\eqref{eq:MM} or the corresponding analogue in the operator framework is
  satisfied; cf.~also Corollaries~\ref{cor:lowerboundOffOp} and~\ref{cor:offForm}\,(a) below. In this regard,
  Theorems~\ref{thm:offdiagOp} and~\ref{thm:offdiagForm} then can also be deduced from~\cite{DES00b} and~\cite{MM15},
  respectively. However, if $A$ does not have such a gap, it is a priori not clear how to derive the two theorems
  from~\cite{DES00b,MM15}.
\end{remark}

The rest of this note is organized as follows. In Section~\ref{sec:applications} we discuss applications of the main theorems and
revisit the Stokes operator as an example in the framework of Theorem~\ref{thm:offdiagForm}. It is also explained there how the
framework for off-diagonal perturbations in Theorem~\ref{thm:offdiagOp} can be applied to general, not necessarily off-diagonal,
perturbations by decomposing the perturbation into its diagonal and off-diagonal parts.
Section~\ref{sec:abstrMinimax} is devoted to an abstract minimax principle based on~\cite{GLS99}.
Two approaches are then used to verify the hypotheses of this abstract minimax principle, the~\emph{graph norm approach} and
the~\emph{block diagonalization approach}, respectively, which are discussed separately in Sections~\ref{sec:graphNorm}
and~\ref{sec:blockDiag} below.
Theorem~\ref{thm:genOpInfinitesimal} is proved in Section~\ref{sec:graphNorm}, which is based on the author's appendix
to~\cite{NSTTV20} and extends the corresponding considerations to certain unbounded perturbations.
Theorems~\ref{thm:genSemibounded}--\ref{thm:offdiagForm} are proved in
Section~\ref{sec:blockDiag}, which builds upon recent developments on block diagonalization of operators and forms
from~\cite{MSS16} and~\cite{GKMSV17}, respectively.
Finally, Appendix~\ref{sec:GLS} reproduces the proof from~\cite{GLS99} for the abstract minimax principle discussed in
Section~\ref{sec:abstrMinimax}, and Appendix~\ref{sec:heinz} provides some consequences of the well-known Heinz inequality that
are used at various spots in this work and are probably folklore.

\section{Applications and examples}\label{sec:applications}

In this section, we use the main results from Section~\ref{sec:intro} to prove monotonicity and continuity properties of minimax
values in gaps of the essential spectrum in various situations and also revisit the well-known Stokes operator in the framework
of Theorem~\ref{thm:offdiagForm} as an example. We finally discuss how to apply the off-diagonal framework from
Theorem~\ref{thm:offdiagOp} to general, not necessarily off-diagonal, perturbations.

We first consider the situation of indefinite or semidefinite bounded perturbations, which has partially been discussed in a
slightly different form in~\cite{NSTTV20}.
For a bounded self-adjoint operator $V$ we define bounded nonnegative operators $V^{(p)}$ and $V^{(n)}$ with
$V = V^{(p)} - V^{(n)}$ via functional calculus by
\begin{equation}\label{eq:defVpVn}
  V^{(p)} := (1 + \sign(V))V / 2
  ,\quad
  V^{(n)} := (\sign(V) - 1)V / 2
  .
\end{equation}
We clearly have $\norm{ V^{(p)} } \le \norm{V}$ and $\norm{ V^{(n)} } \le \norm{V}$.

The following result can be proved in several ways. The proof below is based on Theorem~\ref{thm:genOpInfinitesimal} and is in
its core close to the proofs of Theorems~3.14 and~3.15 in~\cite{NSTTV20}. An alternative proof for part~(b) based on
Theorem~\ref{thm:offdiagOp} is discussed after Remark~\ref{rem:gap} below. The result itself extends Theorem~5 in~\cite{GS99},
which was formulated there for bounded nonpositive $V$ that are relatively compact with respect to the unperturbed operator $A$.

\begin{proposition}\label{prop:boundedPert}
  Let the finite interval $(c,d)$ belong to the resolvent set of the self-adjoint operator $A$, and let $V$ be a bounded
  self-adjoint operator on the same Hilbert space satisfying $\norm{V^{(p)}} + \norm{V^{(n)}} < d - c$ with $V^{(p)}$ and
  $V^{(n)}$ as in~\eqref{eq:defVpVn}. Set $P := \EE_A([d,\infty))$ and $Q := \EE_{A+V}([d-\norm{V^{(n)}},\infty))$. Then:
  \begin{enumerate}
    \renewcommand{\theenumi}{\alph{enumi}}

    \item
    The interval $(c+\norm{V^{(p)}} , d-\norm{V^{(n)}})$ belongs to the resolvent set of the operator $A + V$, and we have
    $\norm{ P - Q } < 1$.

    \item
    With $\cD_+ := \Ran P \cap \Dom(A)$ and $\cD_- := \Ran (I-P) \cap \Dom(A)$ we have
    \begin{equation*}
      \lambda_k\bigl ((A+V)|_{\Ran Q} \bigr)
      =
      \inf_{\substack{\fM_+ \subset \cD_+\\ \dim\fM_+ = k}} \sup_{\substack{x \in \fM_+ \oplus \cD_-\\ \norm{x}=1}}
        \scprod{ x , (A+V)x }
    \end{equation*}
    for all $k \in \NN$ with $k \le \dim \Ran Q$.

    \item
    With $\cE_+ := \Ran Q \cap \Dom(A)$ and $\cE_- := \Ran (I-Q) \cap \Dom(A)$ we have
    \begin{equation*}
      \lambda_k\bigl (A|_{\Ran P} \bigr)
      =
      \inf_{\substack{\fM_+ \subset \cE_+\\ \dim\fM_+ = k}} \sup_{\substack{x \in \fM_+ \oplus \cE_-\\ \norm{x}=1}}
        \scprod{ x , Ax }
    \end{equation*}
    for all $k \in \NN$ with $k \le \dim \Ran P$.

  \end{enumerate}  
\end{proposition}

\begin{proof}
  (a).
  This is Proposition~2.1 and Theorem~1.1 in~\cite{Seel20}, respectively; cf.~also~\cite[Theorem~3.2]{Ves08}. More precisely, the
  variant of the Davis-Kahan $\sin2\Theta$ theorem in~\cite[Theorem~1.1]{Seel20} gives
  \begin{equation*}
  		\norm{ P - Q }
  		\le
  		\sin\Bigl( \frac{1}{2} \arcsin \frac{\norm{V^{(p)}} + \norm{V^{(n)}}}{d-c} \Bigr)
  		<
  		\frac{\sqrt{2}}{2}
  		<
  		1
    .
  \end{equation*}
  In particular, the subspaces $\Ran P$ and $\Ran Q$ have the same dimension; cf.~also Remark~\ref{rem:PQbij}\,(1) below.

  (b).
  Pick $\gamma \in (c + \norm{V^{(p)}} , d - \norm{V^{(n)}})$. By part~(a) we then have $\EE_{A-\gamma}((0,\infty)) = P$ and
  $\EE_{A+V-\gamma}((0,\infty)) = Q$. Moreover, for $x \in \cD_-$ we have
  \begin{equation}\label{eq:boundPertGap}
   \begin{aligned}
    \scprod{ x , (A+V-\gamma)x }
    &=
    \scprod{ x , (A-\gamma)x } + \scprod{ x , V^{(p)}x } - \scprod{ x , V^{(n)}x }\\
    &\le
    (c - \gamma + \norm{V^{(p)}}) \norm{x}^2
    <
    0
    .
   \end{aligned}
  \end{equation}
  In light of part~(a), the claim now follows from Theorem~\ref{thm:genOpInfinitesimal} with $P_+ = P$ and $Q_+ = Q$ upon a
  spectral shift by $\gamma$.

  (c).
  Similarly as in (b), pick $\gamma \in (c + \norm{V^{(p)}} + \norm{V^{(n)}} , d)$. We then have $\EE_{A-\gamma}((0,\infty)) = P$
  and for some chosen $\rho \in (c + \norm{V^{(p)}} , d - \norm{V^{(n)}})$ also $\EE_{A+V-\rho}((0,\infty)) = Q$ . Moreover, for
  $x \in \cE_-$ we have
  \begin{equation}\label{eq:boundPertSwitchedGap}
   \begin{aligned}
    \scprod{ x , (A-\gamma)x }
    &=
    \scprod{ x , (A+V-\gamma)x } - \scprod{ x , V^{(p)}x } + \scprod{ x , V^{(n)}x }\\
    &\le
    (c + \norm{V^{(p)}} - \gamma + \norm{V^{(n)}}) \norm{x}^2
    <
    0
    .
   \end{aligned}
  \end{equation}
  In light of part~(a), the claim now follows analogously from Theorem~\ref{thm:genOpInfinitesimal} with switched roles of $A$
  and $A+V$ and with $P_+ = Q$ and $Q_+ = P$.
\end{proof}%

\begin{remark}\label{rem:boundedPert}
  (1)
  A corresponding representation of the minimax values in terms of the forms associated with $A+V$ and $A$, respectively, as in
  Theorems~\ref{thm:genOpInfinitesimal}--\ref{thm:offdiagOp} holds here as well. However, for the sake of simplicity and since
  this is not needed in Corollaries~\ref{cor:monotonicity} and~\ref{cor:continuity} below, this has not been formulated in
  Proposition~\ref{prop:boundedPert}.

  (2)
  Part~(b) of Proposition~\ref{prop:boundedPert} can also be deduced from~\cite{DES00b}. Indeed, for $y \in \cD_+$ we analogously
  have
  \begin{equation*}
    \scprod{ y , (A+V-\gamma)y }
    \geq
    (d - \gamma - \norm{V^{(n)}}) \norm{y}^2
    >
    0
    ,    
  \end{equation*}
  which together with~\eqref{eq:boundPertGap} implies that the operator analogue to condition~\eqref{eq:MM} is satisfied.
  However, the situation is less clear for part~(c) of Proposition~\ref{prop:boundedPert}. Here, for $y \in \cE_+$ we get
  \begin{equation*}
    \scprod{ y , (A-\gamma)y }
    \geq
    (d - \norm{V^{(n)}} - \gamma - \norm{V^{(p)}}) \norm{y}^2
    ,
  \end{equation*}
  which is positive only if $\gamma < d - \norm{V^{(p)}} - \norm{V^{(n)}}$. Together with the condition
  $\gamma > c + \norm{V^{(p)}} + \norm{V^{(n)}}$ in~\eqref{eq:boundPertSwitchedGap}, this requires the~\emph{stronger} assumption
  $\norm{V^{(p)}} + \norm{V^{(n)}} < (d-c)/2$. The latter can be somehow remedied with a continuity argument, but in the end the
  approach presented in the proof of Proposition~\ref{prop:boundedPert} above is just more convenient.
\end{remark}

The above proposition includes the particular cases where $V$ satisfies $\norm{V} < (d - c)/2$ and where $V$ is semidefinite with
$\norm{V} < d - c$, which in the context of part~(c) have essentially been discussed in the proofs of Theorems~3.14 and~3.15
in~\cite{NSTTV20}. However, Proposition~\ref{prop:boundedPert} allows also certain indefinite perturbations $V$ with
$(d - c)/2 \le \norm{V} < d - c$ that were not covered before and may thus be used to refine the results in~\cite{NSTTV20}.

As immediate corollaries to Proposition~\ref{prop:boundedPert}, we obtain the following monotonicity and continuity statements
for the minimax values in gaps of the essential spectrum, which, in essence, reproduce particular cases of results
in~\cite{Ves08}.

\begin{corollary}\label{cor:monotonicity}
  Let $A$ be as in Proposition~\ref{prop:boundedPert}, and let $V_0$ and $V_1$ be bounded self-adjoint operators on the same
  Hilbert space satisfying the condition
  $\max\{ \norm{V_0^{(p)}} + \norm{V_0^{(n)}} , \norm{V_1^{(p)}} + \norm{V_1^{(n)}} \} < d - c$.

  If, in addition, $V_0 \le V_1$, then
  \begin{equation*}
    \lambda_k\bigl( (A+V_0)|_{\Ran\EE_{A+V_0}([d-\norm{V_0^{(n)}},\infty))} \bigr)
    \le
    \lambda_k\bigl( (A+V_1)|_{\Ran\EE_{A+V_1}([d-\norm{V_1}^{(n)},\infty))} \bigr)
  \end{equation*}
  for $k \le \dim\Ran\EE_A([d,\infty)) = \dim\Ran\EE_{A+V_j}([d-\norm{V_j^{(n)}},\infty))$, $j\in\{0,1\}$.
\end{corollary}

\begin{corollary}\label{cor:continuity}
  Let $A$ and $V$ be as in Proposition~\ref{prop:boundedPert}. Then, the interval $(c+\norm{V^{(p)}} , d-\norm{V^{(n)}})$ belongs
  to the resolvent set of every $A+tV$, $t \in [0,1]$, and for each
  $k \le \dim \Ran \EE_A([d,\infty)) = \dim \Ran \EE_{A+tV}([d-t\norm{V^{(n)}},\infty))$, $t \in [0,1]$, the mapping
  \begin{equation*}
    [0,1]
    \ni
    t
    \mapsto
    \lambda_k((A+tV)|_{\Ran\EE_{A+tV}([d-t\norm{V^{(n)}},\infty))})
  \end{equation*}
  is Lipschitz continuous with Lipschitz constant $\norm{V}$.
\end{corollary}

\begin{proof}
  Taking into account that
  \begin{equation*}
    \scprod{ x , (A+sV)x } - \abs{t-s}\norm{V}
    \le
    \scprod{ x , (A+tV)x }
    \le
    \scprod{ x , (A+sV)x } + \abs{t-s}\norm{V}
  \end{equation*}
  for all $x \in \Dom(A)$, the claim follows immediately from Proposition~\ref{prop:boundedPert}.
\end{proof}%

It should again be mentioned that the above statements include the particular cases where the norm of the perturbations is less
than $(d - c)/2$ or where the perturbations are semidefinite with a norm less than $d - c$. These cases have essentially been
discussed in~\cite{NSTTV20}. There, especially lower bounds on the movement of eigenvalues in gaps of the essential spectrum
under certain conditions and the behaviour of edges of the essential spectrum have been studied. However, since this is not the
main focus of the present work, this is not pursued further here.

As a consequence of Theorem~\ref{thm:offdiagOp}, we obtain the following lower bound for the minimax values in the setting
of off-diagonal operator perturbations.
\begin{corollary}\label{cor:lowerboundOffOp}
  In the situation of Theorem~\ref{thm:offdiagOp}, we have
  \begin{equation*}
    \lambda_k(A|_{\Ran P_+})
    \le
    \lambda_k(B|_{\Ran Q_+})
  \end{equation*}
  for all $k \le \dim\Ran P_+ = \dim \Ran Q_+$.
\end{corollary}

\begin{proof}
  Let $\fM_+ \subset \cD_+$ with $\dim\fM_+ = k$. Since $\scprod{ x , Vx } = 0$ for all $x \in \cD_+$ by hypothesis, we have
  \begin{equation*}
    \sup_{\substack{x \in \fM_+\\ \norm{x}=1}} \scprod{ x , Ax }
    =
    \sup_{\substack{x \in \fM_+\\ \norm{x}=1}} \scprod{ x , (A+V)x }
    \le
    \sup_{\substack{x \in \fM_+ \oplus \cD_-\\ \norm{x}=1}} \scprod{ x , (A+V)x }
    .
  \end{equation*}
  Taking the infimum over all such subspaces $\fM_+$ proves the claim by Theorem~\ref{thm:offdiagOp} and the standard minimax
  values for $A|_{\Ran P_+}$.
\end{proof}%

As in Corollary~\ref{cor:continuity}, we also obtain a continuity statement in the situation of Theorem~\ref{thm:offdiagOp} with
bounded off-diagonal perturbations. Here, however, we do not have to impose any condition on the norm of the perturbation.

\begin{corollary}\label{cor:continuityOffOp}
  Let $A$ and $V$ be as in Theorem~\ref{thm:offdiagOp}, and suppose that $V$ is bounded. Then, for each
  $k \le \dim \Ran\EE_A((0,\infty)) = \dim \Ran\EE_{A+tV}((0,\infty))$, $t \in \RR$, the mapping
  \begin{equation*}
    \RR
    \ni
    t
    \mapsto
    \lambda_k\bigl( (A+tV)|_{\Ran\EE_{A+tV}((0,\infty))} \bigr)
  \end{equation*}
  is Lipschitz continuous with Lipschitz constant $\norm{V}$.
\end{corollary}

In the particular case where $B$ is semibounded, Theorem~\ref{thm:offdiagForm} allows us to extend
Corollaries~\ref{cor:lowerboundOffOp} and~\ref{cor:continuityOffOp} to some degree to off-diagonal form perturbations. Recall
here, that semiboundedness of $B$ implies that also $A$ is semibounded, see the proof of Theorem~\ref{thm:offdiagForm} below.

\begin{corollary}\label{cor:offForm}
  Assume the hypotheses of Theorem~\ref{thm:offdiagForm}.
  \begin{enumerate}
    \renewcommand{\theenumi}{\alph{enumi}}

    \item For each $k \in \NN$ with $k \le \dim\Ran P_+ = \dim\Ran Q_+$ one has
          $\lambda_k(A|_{\Ran P_+}) \le \lambda_k(B|_{\Ran Q_+})$.

    \item Denote for $t \in (-1/b , 1/b)$ by $B_t$ the self-adjoint operator associated with the form $\fb_t := \fa + t\fv$ with
          form domain $\Dom[\fb_t] := \Dom[\fa]$. Then, for each
          $k \le \dim\Ran\EE_A((0,\infty)) = \dim\Ran\EE_{B_t}((0,\infty))$, the mapping
          \begin{equation*}
            ( -1/b , 1/b )
            \ni
            t
            \mapsto
            \lambda_k(B_t|_{\Ran\EE_{B_t}((0,\infty))})
          \end{equation*}
          is locally Lipschitz continuous.

  \end{enumerate}
\end{corollary}

\begin{proof}
  (a).
  Taking into account that $\fv[ x , x ] = 0$ for all $x \in \fD_+$ by hypothesis, the inequality
  $\lambda_k(A|_{\Ran P_+}) \le \lambda_k(B|_{\Ran Q_+})$ is proved by means of Theorem~\ref{thm:offdiagForm} in a way analogous
  to Corollary~\ref{cor:lowerboundOffOp}.

  (b).
  Recall that each $B_t$ is indeed a semibounded self-adjoint operator with $\Dom[\fb_t] = \Dom(\abs{B_t}^{1/2})$ by the
  well-known KLMN theorem, and note that each $t\fv$ satisfies the hypotheses of Theorem~\ref{thm:offdiagForm}.
  Pick $t,s \in (-1/b , 1/b)$ with $b\abs{t-s} \le 1-b\abs{s}$.
  
  Consider first the case where $A$ (and hence $\fa$) is lower semibounded with lower bound $m \in \RR$. We then have
  $\abs{\fa[x,x]} \le \fa[x,x] + (\abs{m}-m)\norm{x}^2$ for all $x \in \Dom[\fa]$. With $\tilde{a} := a + b\abs{m} - bm$, this
  gives
  \begin{equation*}
    \abs{\fv[x,x]}
    \le
    \tilde{a}\norm{x}^2 + b\fa[x,x]
    \le
    \tilde{a}\norm{x}^2 + b\fb_s[x,x] + b\abs{s}\abs{\fv[x,x]}
  \end{equation*}
  and, hence,
  \begin{equation*}
    \abs{\fv[x,x]}
    \le
    \frac{\tilde{a}}{1-b\abs{s}}\norm{x}^2 + \frac{b}{1-b\abs{s}}\fb_s[x,x]
  \end{equation*}
  for all $x \in \Dom[\fa] = \Dom[\fb_s]$. Since $\fb_t = \fb_s + (t-s)\fv$, we thus obtain
  \begin{equation*}
    -\frac{\tilde{a}\abs{t-s}}{1-b\abs{s}} + \Bigl( 1 - \frac{b\abs{t-s}}{1-b\abs{s}} \Bigr) \fb_s
    \le
    \fb_t
    \le
    \frac{\tilde{a}\abs{t-s}}{1-b\abs{s}} + \Bigl( 1 + \frac{b\abs{t-s}}{1-b\abs{s}} \Bigr) \fb_s
    .
  \end{equation*}
  Abbreviating $\lambda_k(t):=\lambda_k(B_t|_{\Ran\EE_{B_t}((0,\infty))})$, Theorem~\ref{thm:offdiagForm} then implies that
  \begin{equation*}
    -\frac{\tilde{a}\abs{t-s}}{1-b\abs{s}} + \Bigl( 1 - \frac{b\abs{t-s}}{1-b\abs{s}} \Bigr) \lambda_k(s)
    \le
    \lambda_k(t)
    \le
    \frac{\tilde{a}\abs{t-s}}{1-b\abs{s}} + \Bigl( 1 + \frac{b\abs{t-s}}{1-b\abs{s}} \Bigr) \lambda_k(s)
  \end{equation*}
  and, therefore,
  \begin{equation}\label{eq:Lipschitz}
    \abs{\lambda_k(t) - \lambda_k(s)}
    \le
    \frac{\tilde{a}\abs{t-s}}{1-b\abs{s}} + \frac{b\abs{t-s}}{1-b\abs{s}} \abs{\lambda_k(s)}
    .
  \end{equation}
  This proves that $t \mapsto \lambda_k(t)$ is continuous on $(-1/b,1/b)$ and, in particular, bounded on every compact
  subinterval of $(-1/b , 1/b)$. In turn, it then easily follows from~\eqref{eq:Lipschitz} that this mapping is even locally
  Lipschitz continuous, which concludes the case where $A$ is lower semibounded.

  If $A$ is upper semibounded with upper bound $m \in \RR$, we proceed similarly. We then have
  $\abs{\fa[x,x]} \le -\fa[x,x] + (m+\abs{m})\norm{x}^2$ for all $x \in \Dom[\fa]$. With $\tilde{a} := a + bm + b\abs{m}$, this
  leads to
  \begin{equation*}
    \abs{\fv[x,x]}
    \le
    \frac{\tilde{a}}{1-b\abs{s}}\norm{x}^2 - \frac{b}{1-b\abs{s}}\fb_s[x,x]
  \end{equation*}
  for all $x \in \Dom[\fa] = \Dom[\fb_s]$. Analogously as above, we then eventually obtain again~\eqref{eq:Lipschitz}, which
  proves the claim in the case where $A$ is upper semibounded. This completes the proof.
\end{proof}%

\begin{remark}\label{rem:offForm}
  In part~(a) of Corollary~\ref{cor:offForm}, one can also give an upper bound for $\lambda_k(B|_{\Ran Q_+})$ in terms of the
  form bounds of $\fv$: If $A$ is lower semibounded with lower bound $m \in \RR$, then we have as in the proof of part~(b) of
  Corollary~\ref{cor:offForm} that
  \begin{equation*}
    \abs{ \fv[ x , x ] }
    \le
    (a + b\abs{m} - bm)\norm{x}^2 + b\fa[ x , x ]
  \end{equation*}
  for all $x \in \Dom[\fa]$, leading to
  \begin{equation*}
    \lambda_k(B|_{\Ran Q_+})
    \le
    (1+b)\lambda_k(A|_{\Ran P_+}) + (a + b\abs{m} - bm)
  \end{equation*}
  for all $k \le \dim\Ran P_+ = \dim\Ran Q_+$. Similarly, if $A$ is upper semibounded with upper bound $m \in \RR$, we have
  \begin{equation*}
    \abs{ \fv[ x , x ] }
    \le
    (a + b\abs{m} + bm)\norm{x}^2 - b\fa[ x , x ]
  \end{equation*}
  for all $x \in \Dom[\fa]$. If, in addition, $b \le 1$, this then leads to
  \begin{equation*}
    \lambda_k(B|_{\Ran Q_+})
    \le
    (1-b)\lambda_k(A|_{\Ran P_+}) + (a + b\abs{m} + bm)
  \end{equation*}
  for all $k \le \dim\Ran P_+ = \dim\Ran Q_+$.
\end{remark}

\subsection*{An example. The Stokes operator}
We now briefly revisit the Stokes operator in the framework of Theorem~\ref{thm:offdiagForm}. Here, we mainly rely
on~\cite{GKMSV19}, but the reader is referred also to~\cite[Section~7]{GKMSV17},~\cite[Chapter~5]{SchmDiss},
~\cite{FFMM00}, and the references cited therein.

Let $\Omega \subset \RR^n$, $n \ge 2$, be a bounded domain with $C^2$-boundary, and let $\nu > 0$ and $v_* \ge 0$. On the Hilbert
space $\cH = \cH_+ \oplus \cH_-$ with $\cH_+ = L^2(\Omega)^n$ and $\cH_- = L^2(\Omega)$, we consider the closed, densely defined,
and nonnegative form $\fa$ with $\Dom[\fa] := H_0^1(\Omega)^n \oplus L^2(\Omega)$ and
\begin{equation*}
  \fa[ v \oplus q , u \oplus p ]
  :=
  \nu \sum_{j=1}^n \int_\Omega \scprod{ \partial_j v(x) , \partial_j u(x) }_{\CC^n} \,\dd x
\end{equation*}
for $u \oplus p, v \oplus q \in \Dom[\fa]$. Clearly, $\fa$ is the form associated to the nonnegative self-adjoint operator
$A := -\nu\mathbf{\Delta} \oplus 0$ on the Hilbert space $\cH = \cH_+ \oplus \cH_-$ with
$\Dom(A) := (H^2(\Omega) \cap H_0^1(\Omega))^n \oplus L^2(\Omega)$ and $\Dom(\abs{A}^{1/2}) = \Dom[\fa]$, where
$\mathbf{\Delta} = \Delta\cdot I_{\CC^n}$ is the vector-valued Dirichlet Laplacian on $\Omega$. Moreover,
$P_+ := \EE_A((0,\infty))$ and $P_- := \EE_A((-\infty,0]) = \EE_A(\{0\})$ are the orthogonal projections onto $\cH_+$ and
$\cH_-$, respectively. In particular, we have
\begin{equation*}
  \fD_+
  :=
  \Ran P_+ \cap \Dom(\abs{A}^{1/2})
  =
  H_0^1(\Omega)^n \oplus 0
\end{equation*}
and
\begin{equation*}
  \fD_-
  :=
  \Ran P_- \cap \Dom(\abs{A}^{1/2})
  =
  0 \oplus L^2(\Omega)
  .
\end{equation*}

Define the symmetric sesquilinear form $\fv$ on $\cH = \cH_+ \oplus \cH_-$ with domain $\Dom[\fv] := \Dom[\fa]$ by
\begin{equation*}
  \fv[ v \oplus q , u \oplus p ]
  :=
  -v_* \scprod{ \divgc v , p }_{L^2(\Omega)} - v_* \scprod{ q , \divgc u }_{L^2(\Omega)}
\end{equation*}
for $u \oplus p, v \oplus q \in \Dom[\fa]$. One can show that
$\nu\norm{\divgc u}_{L^2(\Omega)}^2 \le \fa[u \oplus 0 , u \oplus 0]$ for all $u \in \fD_+ = H_0^1(\Omega)^n$, see,
e.g.,~\cite[Proof of Theorem~5.12]{SchmDiss}. Using Young's inequality, this then implies that $\fv$ is infinitesimally form
bounded with respect to $\fa$, see~\cite[Remark~5.1.3]{SchmDiss}; cf.~also~\cite[Section~2]{GKMSV19}. Indeed, for $\eps > 0$ and
$f = u \oplus p \in \Dom[\fa]$ we obtain
\begin{equation}\label{eq:StokesRelBound}
 \begin{aligned}
  \abs{\fv[f , f]}
  &\le
  2 v_* \abs{ \scprod{ p , \divgc u }_{L^2(\Omega)} }
    \le
    2 v_* \norm{p}_{L^2(\Omega)} \norm{\divgc u}_{L^2(\Omega)}\\
  &\le
  \eps \nu \norm{\divgc u}_{L^2(\Omega)}^2 + \eps^{-1} \nu^{-1} v_*^2 \norm{p}_{L^2(\Omega)}^2\\
  &\le
  \eps \fa[ u \oplus 0 , u \oplus 0 ] + \eps^{-1} \nu^{-1} v_*^2 \norm{f}_\cH^2\\
  &=
  \eps \fa[ f , f ] + \eps^{-1} \nu^{-1} v_*^2 \norm{f}_\cH^2
  .
 \end{aligned}
\end{equation}
Thus, by the well-known KLMN theorem, the form $\fb_S := \fa + \fv$ with $\Dom[\fb_S] = \Dom[\fa] = \Dom(\abs{A}^{1/2})$ is
associated to a unique lower semibounded self-adjoint operator $B_S$ on $\cH$ with $\Dom(\abs{B_S}^{1/2}) = \Dom(\abs{A}^{1/2})$,
the so-called~\emph{Stokes operator}. It is a self-adjoint extension of the (non-closed) upper dominant block operator matrix
\begin{equation*}
  \begin{pmatrix}
    -\nu\mathbf{\Delta} & v_*\grad\\
    -v_*\divgc & 0
  \end{pmatrix}
\end{equation*}
defined on $(H^2(\Omega) \cap H_0^1(\Omega))^n \oplus H^1(\Omega)$. In fact, the closure of the latter is a self-adjoint
operator, see~\cite[Theorems~3.7 and~3.9]{FFMM00}, which yields another characterization of the Stokes operator $B_S$.

By rescaling, one obtains from~\cite[Theorem~3.15]{FFMM00} that the essential spectrum of $B_S$ is given by
\begin{equation*}
  \spec_\ess(B_S)
  =
  \Bigl\{ -\frac{v_*^2}{\nu} , -\frac{v_*^2}{2\nu}  \Bigr\}
  ,
\end{equation*}
cf.~\cite[Remark~2.2]{GKMSV19}. In particular, the essential spectrum of $B_S$ is purely negative. In turn, the positive spectrum
of $B_S$, that is, $\spec(B_S) \cap (0,\infty)$, is discrete~\cite[Theorem~2.1\,(i)]{GKMSV19}.

The above shows that the hypotheses of Theorem~\ref{thm:offdiagForm} are satisfied in this situation, so that we obtain from
Theorem~\ref{thm:offdiagForm} and Corollary~\ref{cor:offForm} the following result.

\begin{proposition}\label{prop:Stokes}
  Let $B_S$ be the Stokes operator as above. Then, the positive spectrum of $B_S$, $\spec(B_S) \cap (0,\infty)$, is discrete, and
  the positive eigenvalues $\lambda_k(B_S|_{\Ran\EE_{B_S}((0,\infty))})$, $k \in \NN$, of $B_S$, enumerated in nondecreasing
  order and counting multiplicities, admit the representation
  \begin{equation*}
    \lambda_k(B_S|_{\Ran\EE_{B_S}((0,\infty))})
    =
    \inf_{\substack{\fM_+ \subset H_0^1(\Omega)^n\\ \dim\fM_+ = k}}
      \sup_{\substack{u \oplus p \in \fM_+ \oplus L^2(\Omega)\\ \norm{u}_{L^2(\Omega)^n}^2 + \norm{p}_{L^2(\Omega)}^2 = 1}}
      \fb_S[ u \oplus p , u \oplus p ]
    .
  \end{equation*}
  The latter depend locally Lipschitz continuously on $\nu$ and $v_*$ and satisfy the two-sided estimate
  \begin{equation*}
    \nu\lambda_k(-\mathbf{\Delta})
    \le
    \lambda_k(B_S|_{\Ran\EE_{B_S}((0,\infty))})
    \le
    \nu\lambda_k(-\mathbf{\Delta}) + \frac{v_*^2}{\nu}
    .
  \end{equation*}
\end{proposition}

\begin{proof}
  In view the above considerations, the representation of the eigenvalues follows from Theorem~\ref{thm:offdiagForm}, and the
  lower bound on the eigenvalues follows from Corollary~\ref{cor:offForm}\,(a). Moreover, by rescaling, the continuity statement is a consequence of
  Corollary~\ref{cor:offForm}\,(b). It remains to show the upper bound on the eigenvalues. To this end, let
  $\fM_+ \subset H_0^1(\Omega)^n$ with $\dim \fM_+ = k \in \NN$, and let $f = u \oplus p \in \fM_+ \oplus L^2(\Omega)$ be a
  normalized vector with $u \neq 0$. Then,
  $\mu := \fa[ u \oplus 0 , u \oplus 0 ] / \norm{u}_{L^2(\Omega)^n}^2 = \fa[ f , f ] / \norm{u}_{L^2(\Omega)^n}^2$ is positive
  and satisfies
  \begin{equation}\label{eq:muBound}
    \mu
    \le
    \sup_{\substack{v \in \fM_+\\ \norm{v}_{L^2(\Omega)^n}^2 = 1}} \fa[ v \oplus 0 , v \oplus 0 ]
  \end{equation}
  and
  \begin{equation*}
    \frac{\nu\norm{\divgc u}_{L^2(\Omega)}^2}{\mu}
    =
    \frac{\norm{u}_{L^2(\Omega)^n}^2 \nu \norm{\divgc u}_{L^2(\Omega)}^2}{\fa[ u \oplus 0 , u \oplus 0 ]}
    \le
    \norm{u}_{L^2(\Omega)^n}^2
    \le
    1
    .
  \end{equation*}
  Similarly as in~\eqref{eq:StokesRelBound}, we now obtain by means of Young's inequality that
  \begin{align*}
    \abs{ \fv[ f , f ] }
    &\le
    2v_*\norm{p}_{L^2(\Omega)} \norm{\divgc u}_{L^2(\Omega)}
      \le
      \mu \norm{p}_{L^2(\Omega)}^2 + \frac{v_*^2 \norm{\divgc u}_{L^2(\Omega)}^2}{\mu}\\
    &\le
    \mu \norm{p}_{L^2(\Omega)}^2 + \frac{v_*^2}{\nu}
    .
  \end{align*}
  Since $\fa[ f , f ] = \mu\norm{u}_{L^2(\Omega)^n}^2$, this gives
  \begin{equation}\label{eq:StokesFormUpperBound}
    \fb_S[ f , f ]
    \le
    \fa[ f , f ] + \mu \norm{p}_{L^2(\Omega)}^2 + \frac{v_*^2}{\nu}
    =
    \mu + \frac{v_*^2}{\nu}
    .
  \end{equation}
  In light of $\fb_S[ 0 \oplus p , 0 \oplus p ] = \fa[ 0 \oplus p , 0 \oplus p ] = 0$, we conclude from~\eqref{eq:muBound}
  and~\eqref{eq:StokesFormUpperBound} that
  \begin{equation*}
    \sup_{\substack{u \oplus p \in \fM_+ \oplus L^2(\Omega)\\ \norm{u}_{L^2(\Omega)^2}^2 + \norm{p}_{L^2(\Omega)}^2 = 1}}
      \fb_S[ u \oplus p , u \oplus p ]
    \le
    \sup_{\substack{v \in \fM_+\\ \norm{v}_{L^2(\Omega)^n}^2 = 1}} \fa[ v \oplus 0 , v \oplus 0 ] + \frac{v_*^2}{\nu}
    ,
  \end{equation*}
  and taking the infimum over subspaces $\fM_+ \subset H_0^1(\Omega)^n$ with $\dim\fM_+ = k$ proves the upper bound. This
  completes the proof.
\end{proof}%

\begin{remark}
  (1)
  Choosing $\eps = 1$ in~\eqref{eq:StokesRelBound}, the upper bound from Remark~\ref{rem:offForm}\,(1) reads
  \begin{equation*}
    \lambda_k(B_S|_{\Ran\EE_{B_S}((0,\infty))})
    \le
    2\nu\lambda_k(-\mathbf{\Delta}) + \frac{v_*^2}{\nu}
  \end{equation*}
  for all $k \in \NN$, while the choice $\eps = v_*$ in~\eqref{eq:StokesRelBound} leads to
  \begin{equation*}
    \lambda_k(B_S|_{\Ran\EE_{B_S}((0,\infty))})
    \le
    (1+v_*)\nu\lambda_k(-\mathbf{\Delta}) + \frac{v_*}{\nu}
  \end{equation*}
  for all $k \in \NN$.

  (2)
  For the particular case of $k = 1$, a similar upper bound has been established in the proof
  of~\cite[Theorem~2.1\,(i)]{GKMSV19}:
  \begin{equation*}
    \nu\lambda_1(-\Delta)
    \le
    \lambda_1(B_S|_{\Ran\EE_{B_S}((0,\infty))})
    \le
    \nu\lambda_1(-\Delta) + v_*\norm{\divgc u_0}_{L^2(\Omega)}
    ,
  \end{equation*}
  where $u_0 \in (H^2(\Omega) \cap H_0^1(\Omega))^n$ is a normalized eigenfunction for $-\mathbf{\Delta}$ corresponding to the
  first positive eigenvalue $\lambda_1(-\mathbf{\Delta}) = \lambda_1(-\Delta)$.
\end{remark}

\begin{remark}\label{rem:StokesMMSST}
  Since $B_S$ is lower semibounded, Proposition~\ref{prop:Stokes} can alternatively be proved via~\cite{GS99}, see
  Remark~\ref{rem:offdiagForm}\,(1). Moreover, since $A = -\nu\mathbf{\Delta} \oplus 0$ has a spectral gap to the right of $0$,
  the same is true with~\cite{MM15}, see Remark~\ref{rem:offdiagForm}\,(2). In fact,~\cite{SST20} gives for
  $\lambda_k = \lambda_k(B_S|_{\Ran\EE_{B_S}((0,\infty))})$, $k \in \NN$, with the same reasoning also the representation
  \begin{equation*}
    \lambda_k
    =
    \inf_{\substack{\fM_+ \subset (H^2(\Omega)\cap H_0^1(\Omega))^n\\ \dim\fM_+ = k}}
      \sup_{\substack{u \oplus p \in \fM_+ \oplus H^1(\Omega)\\ \norm{u}_{L^2(\Omega)^n}^2 + \norm{p}_{L^2(\Omega)}^2 = 1}}
      \langle u \oplus p , B_S(u \oplus p) \rangle
    .
  \end{equation*}
\end{remark}

\subsection*{Reducing to the off-diagonal framework}
Apart from the Stokes operator from the previous subsection, the consideration of off-diagonal perturbations in
Theorems~\ref{thm:offdiagOp} and~\ref{thm:offdiagForm} may seem a bit restrictive. However, such perturbations naturally appear
when the perturbation is decomposed into its diagonal and off-diagonal part. If the diagonal perturbation can then be handled
efficiently in a suitable way, the discussed off-diagonal framework can be applied to the remaining part of the perturbation. The
following result makes this precise in the setting of operator perturbations.

\begin{proposition}\label{prop:redOffDiag}
  Assume Hypothesis~\ref{hyp:minimax}. Suppose, in addition, that $B$ is of the form $B = A + V$, $\Dom(B) = \Dom(A)$, with some
  symmetric $A$-bounded operator $V$. Denote $A_\pm := A|_{\Ran P_\pm}$, and decompose $V|_{\Dom(A)}$ as
  $V|_{\Dom(A)} = V_{\mathrm{diag}} + V_{\mathrm{off}}$, where $V_{\mathrm{diag}} = V_+ \oplus V_-$ is the diagonal part of
  $V|_{\Dom(A)}$ and $V_{\mathrm{off}}$ is the off-diagonal part of $V|_{\Dom(A)}$ with respect to the decomposition
  $\Ran P_+ \oplus \Ran P_-$. Suppose that the following hold:
  \begin{enumerate}
    \renewcommand{\theenumi}{\roman{enumi}}

    \item
    $A+V_{\mathrm{diag}}$ is self-adjoint on $\Dom(A+V_{\mathrm{diag}}) = \Dom(A)$,

    \item
    $V_{\mathrm{off}}$ is $(A+V_{\mathrm{diag}})$-bounded with $(A+V_{\mathrm{diag}})$-bound smaller than $1$,

    \item
    $\sup\spec(A_- + V_-) \leq \inf\spec(A_+ + V_+)$, and

    \item
    $\Ker(A_+ + V_+ - \mu) = \{ 0 \}$ with $\mu := \sup\spec(A_- + V_-)$.

  \end{enumerate}
  Then, setting $Q := \EE_B((\mu,\infty))$, one has $\dim\Ran P_+ = \dim\Ran Q$ and
  \begin{equation*}
    \lambda_k(B|_{\Ran Q})
    =
    \inf_{\substack{\fM_+\subset\cD_+\\ \dim\fM_+=k}} \sup_{\substack{x\in\fM_+ \oplus \cD_-\\ \norm{x}=1}} \scprod{ x , Bx }
    =
    \inf_{\substack{\fM_+\subset\fD_+\\ \dim\fM_+=k}} \sup_{\substack{x\in\fM_+ \oplus \fD_-\\ \norm{x}=1}} \fb[x,x]
  \end{equation*}
  for all $k \in \NN$ with $k \le \dim\Ran Q$.
\end{proposition}

\begin{proof}
  By (iii) and (iv) we obviously have
  \begin{equation*}
    P_+ = \EE_{A+V_{\text{diag}}}\bigl( (\mu,\infty) \bigr)
    \quad\text{ and }\quad
    P_- = \EE_{A+V_{\text{diag}}}\bigl( (-\infty,\mu] \bigr)
    .
  \end{equation*}
  Upon a spectral shift by $\mu$, the claim now follows from Theorem~\ref{thm:offdiagOp} with $A$, $V$, and $Q_+$ replaced by
  $A + V_{\text{diag}}$, $V_{\text{off}}$, and $Q$, respectively.
\end{proof}%

\begin{remark}\label{rem:gap}
  Conditions (iii) and (iv) in Proposition~\ref{prop:redOffDiag} are satisfied if
  $\mu < \nu := \inf\spec(A_+ + V_+)$ holds. In this case, the interval $(\mu , \nu)$ belongs to the resolvent set of the
  operator $A + V_{\text{diag}}$, and by~\cite[Theorem~1]{MS06} (cf.~also~\cite[Theorem~2.1]{AL95}) to the one of
  $A + V = (A + V_{\text{diag}}) + V_{\text{off}}$ as well. In particular, we have
  $\EE_{A+V}((\mu,\infty)) = \EE_{A+V}([\nu,\infty))$. This is the situation encountered in the alternative proof of
  Proposition~\ref{prop:boundedPert}\,(b) and the proof of Corollary~\ref{cor:unboundedPert} below. However, the conclusions can
  then alternatively be obtained also via~\cite{DES00b,MM15}, cf.~Remark~\ref{rem:offdiagForm}\,(2).
\end{remark}

Since conditions~(i) and (ii) in Proposition~\ref{prop:redOffDiag} are clearly satisfied if $V$ is bounded, the above provides an
alternative way to prove part~(b) of Proposition~\ref{prop:boundedPert}:

\begin{proof}[Alternative proof of Proposition~\ref{prop:boundedPert}\,(b)]
  By spectral shift we may assume without loss of generality that $c < 0 < d$. We then have $P = P_+$ with $P_+$ as in
  Hypothesis~\ref{hyp:minimax}. Let $A_\pm$ and $V_{\text{diag}}$ be defined as in Proposition~\ref{prop:redOffDiag}, and for
  $\bullet \in \{ p , n \}$ denote by $V_{\text{diag}}^{(\bullet)} = V_+^{(\bullet)} \oplus V_-^{(\bullet)}$ the diagonal part of
  $V^{(\bullet)}$. Clearly, we have $V_\pm^{(\bullet)} \ge 0$ and $\norm{V_\pm^{(\bullet)}} \le \norm{V^{(\bullet)}}$, and
  $V_{\text{diag}}$ decomposes as
  $V_{\text{diag}} = V_{\text{diag}}^{(p)} - V_{\text{diag}}^{(n)} = (V_+^{(p)} - V_+^{(n)}) \oplus (V_-^{(p)} - V_-^{(n)})$.
  Now,
  \begin{equation*}
    A_- + V_-^{(p)} - V_-^{(n)}
    \le
    c + \norm{V^{(p)}}
    \quad\text{ and }\quad
    d - \norm{V^{(n)}}
    \le
    A_+ + V_+^{(p)} - V_+^{(n)}
  \end{equation*}
  in the sense of quadratic forms. In light of $c + \norm{V^{(p)}} < d - \norm{V^{(n)}}$ and Remark~\ref{rem:gap}, applying
  Proposition~\ref{prop:redOffDiag} proves the claim.
\end{proof}%

It is worth to note that an analogous reasoning for part~(c) of Proposition~\ref{prop:boundedPert} suffers from similar obstacles
as the alternative proof based on~\cite{DES00b,MM15} mentioned in Remark~\ref{rem:boundedPert}\,(2).

If $V$ is not bounded, conditions~(i) and (ii) in Proposition~\ref{prop:redOffDiag} can still be guaranteed via the well-known
Kato-Rellich theorem by means of a sufficiently small $A$-bound of $V$, as the following lemma shows.

\begin{lemma}\label{lem:partsRelBound}
  In the situation of Proposition~\ref{prop:redOffDiag}, let $V$ have $A$-bound smaller than $1/2$. Then $V_{\mathrm{diag}}$ and
  $V_{\mathrm{off}}$ both have $A$-bound smaller than $1/2$, and $V_{\mathrm{off}}$ has $(A+V_{\mathrm{diag}})$-bound smaller
  than $1$.
\end{lemma}

\begin{proof}
  By hypothesis, there are constants $a,b \ge 0$, $b < 1/2$ such that we have $\norm{ Vx } \le a\norm{ x } + b\norm{ Ax }$ for
  all $x \in \Dom(A)$. Using Young's inequality, this gives for every $\eps > 0$ that
  \begin{equation*}
    \norm{ Vx }^2
    \le
    \bigl( a\norm{x} + b\norm{Ax} \bigr)^2
    \le
    a^2\Bigl( 1 + \frac{1}{\eps} \Bigr)\norm{x}^2 + b^2(1+\eps)\norm{Ax}^2
  \end{equation*}
  for all $x \in \Dom(A)$; cf.~\cite[Section~V.4.1]{Kato95}. Since $AP_\pm x = P_\pm Ax$ for all $x \in \Dom(A)$ and the ranges
  of $P_\pm$ are orthogonal, this implies that
  \begin{align*}
    \norm{V_{\text{diag}}x}^2
    &=
    \norm{P_+VP_+x}^2 + \norm{P_-VP_-x}^2\\
    &\le
    a^2\Bigl( 1 + \frac{1}{\eps} \Bigr) \norm{x}^2 + b^2(1+\eps) \norm{Ax}^2
  \end{align*}
  for all $x \in \Dom(A)$. We choose $\eps > 0$ such that $\beta := b(1+\eps)^{1/2} < 1/2$ and set $\alpha := a(1+1/\eps)^{1/2}$.
  It then follows from the above that
  \begin{equation*}
    \norm{V_{\text{diag}}x}
    \le
    \alpha \norm{x} + \beta \norm{Ax}
    \quad\text{ for all }\
    x \in \Dom(A)
  \end{equation*}
  and analogously the same for $V_{\text{off}}$. This shows that $V_{\text{diag}}$ and $V_{\text{off}}$ indeed have $A$-bound
  smaller than $1/2$.

  Using standard arguments as, for instance, in~\cite[Lemma~2.1.6]{Tre08}, we obtain
  \begin{equation*}
    \norm{V_{\text{diag}}x}
    \le
    \frac{\alpha}{1-\beta}\norm{x} + \frac{\beta}{1-\beta}\norm{(A + V_{\text{diag}})x}
    \quad\text{ for all }\
    x \in \Dom(A)
  \end{equation*}
  and, in turn,
  \begin{align*}
    \norm{V_{\text{off}}x}
    &\le
    \alpha\norm{x} + \beta\norm{Ax}
      \le
      \alpha\norm{x} + \beta\norm{(A+V_{\text{diag}})x} + \beta\norm{V_{\text{diag}}x}\\
    &\le
    \Bigl( \alpha + \frac{\alpha\beta}{1-\beta} \Bigr)\norm{x} + \Bigl( \beta + \frac{\beta^2}{1-\beta} \Bigr)
      \norm{(A + V_{\text{diag}})x}\\
    &=
    \frac{\alpha}{1-\beta}\norm{x} + \frac{\beta}{1-\beta}\norm{(A + V_{\text{diag}})x}
  \end{align*}
  for all $x \in \Dom(A + V_{\text{diag}}) = \Dom(A)$, where $\beta/(1-\beta) < 1$. This shows that $V_{\text{off}}$ has
  $(A + V_{\text{diag}})$-bound smaller than $1$ and, hence, completes the proof.
\end{proof}%

A suitable smallness assumption on the perturbation may also be used to guarantee condition (iii) and (iv) in
Proposition~\ref{prop:redOffDiag} in the sense of Remark~\ref{rem:gap} if the unperturbed operator has a gap in the spectrum.
This is demonstrated in the following corollary to the above.

\begin{corollary}\label{cor:unboundedPert}
  Let $A$, $(c,d)$, and $\cD_\pm$ be as in Proposition~\ref{prop:boundedPert}, and let $V$ be a symmetric operator that is
  $A$-bounded with $A$-bound smaller than $1/2$. Suppose, in addition, that $c < 0 < d$, and define $A_\pm$ and $V_\pm$ as in
  Proposition~\ref{prop:redOffDiag}. Suppose that there are constants $a_\pm, b_\pm \ge 0$, $b_\pm < 1$, with
  \begin{equation}\label{eq:relBounds}
    \abs{ \scprod{ x , V_\pm x } }
    \le
    a_\pm \norm{x}^2 \pm b_\pm \scprod{ x , A_\pm x }
    \quad\text{ for all }\
    x \in \cD_\pm
  \end{equation}
  and
  \begin{equation}\label{eq:unboundedPert}
    a_+ + a_- + b_+d - b_-c < d - c.
  \end{equation}
  Then, the interval $(a_- + (1-b_-)c , (1-b_+)d-a_+)$ belongs to the resolvent set of $A+V$, and one has
  $\dim \Ran \EE_A( [d,\infty) ) = \dim \EE_{A+V}( [(1-b_+)d-a_+,\infty) )$ and
  \begin{equation*}
    \lambda_k((A+V)|_{\Ran\EE_{A+V}([(1-b_+)d-a_+,\infty))})
    =
    \inf_{\substack{\fM_+ \subset \cD_+\\ \dim\fM_+ = k}} \sup_{\substack{x\in\fM_+\oplus\cD_-\\ \norm{x}=1}}
      \scprod{ x , (A+V)x }
  \end{equation*}
  for all $k \in \NN$ with $k \le \dim\Ran \EE_{A+V}([(1-b_+)d-a_+,\infty))$.
\end{corollary}

\begin{proof}
  By Lemma~\ref{lem:partsRelBound} and the Kato-Rellich theorem, conditions~(i) and~(ii) in Proposition~\ref{prop:redOffDiag} are
  satisfied. By~\eqref{eq:relBounds} we have
  \begin{equation*}
    A_- + V_-
    \le
    a_- + (1-b_-)c
    \quad\text{ and }\quad
    (1-b_+)d - a_+
    \le
    A_+ + V_+
  \end{equation*}
  in the sense of quadratic forms. Since $a_- + (1-b_-)c < (1-b_+)d - a_+$ by~\eqref{eq:unboundedPert}, the claim now again
  follows from Proposition~\ref{prop:redOffDiag} and Remark~\ref{rem:gap}.
\end{proof}%

\begin{remark}\label{rem:unboundedPert}
  (1)
  It is easy to see that the left-hand side of~\eqref{eq:unboundedPert} is invariant under a spectral shift in $A$. In this
  respect, an analogous statement as in Corollary~\ref{cor:unboundedPert} holds for arbitrary spectral gaps $(c,d)$, not just
  ones satisfying $c < 0 < d$.
   
  (2)
  It follows from Lemma~\ref{lem:partsRelBound} that $V_\pm$ is $A_\pm$-bounded with $A_\pm$-bound smaller than $1/2$. In turn,
  since both $A_\pm$ are semibounded, constants $a_\pm,b_\pm \ge 0$, $b_\pm < 1/2$, satisfying~\eqref{eq:relBounds} always exist
  by~\cite[Theorem~VI.1.38]{Kato95}. Condition~\eqref{eq:unboundedPert} can then be guaranteed for $tV$ instead of $V$ for
  $t \in \RR$ with sufficiently small modulus.

  (3)
  As indicated in Remark~\ref{rem:gap}, Corollary~\ref{cor:unboundedPert} can also be proved via~\cite{DES00b,MM15} by
  verifying the operator analogue to~\eqref{eq:MM}. In fact, that approach even allows to weaken the assumption on the $A$-bound
  of $V$ from being smaller than $1/2$ to merely being smaller than $1$ since it is then not necessary to have that
  $V_{\text{off}}$ has $(A+V_{\text{diag}})$-bound smaller than $1$.
\end{remark}

\section{An abstract minimax principle in spectral gaps}\label{sec:abstrMinimax}

We rely on the following abstract minimax principle in spectral gaps, part~(a) of which is extracted from~\cite{GLS99} and
part~(b) of which is its natural adaptation to the operator framework; cf.~also~\cite[Proposition~A.3]{NSTTV20}. For the
convenience of the reader, its proof is reproduced in Appendix~\ref{sec:GLS} below.

\begin{proposition}\label{prop:GLS}
  Assume Hypothesis \ref{hyp:minimax}.
  \begin{enumerate}
    \renewcommand{\theenumi}{\alph{enumi}}

    \item
    If we have $\Dom(\abs{B}^{1/2}) = \Dom(\abs{A}^{1/2})$, $\fb[x,x]\le0$ for all $x\in\fD_-$, and
    $\Ran(P_+Q_+|_{\fD_+}) \supset \fD_+$, then
    \begin{equation*}
      \lambda_k(B|_{\Ran Q_+})
      =
      \inf_{\substack{\fM_+\subset\fD_+\\ \dim(\fM_+)=k}} \sup_{\substack{x\in\fM_+\oplus\fD_-\\ \norm{x}=1}} \fb[x,x]
    \end{equation*}
    for all $k \in \NN$ with $k \le \dim\Ran P_+$.

    \item
    If we have $\Dom(B) = \Dom(A)$, $\scprod{ x , Bx } \le 0$ for all $x \in \cD_-$, and $\Ran(P_+Q_+|_{\cD_+}) \supset \cD_+$,
    then
    \begin{equation*}
      \lambda_k(B|_{\Ran Q_+})
      =
      \inf_{\substack{\fM_+\subset\cD_+\\ \dim(\fM_+)=k}} \sup_{\substack{x\in\fM_+\oplus\cD_-\\ \norm{x}=1}}
        \scprod{ x , Bx }
    \end{equation*}
    for all $k \in \NN$ with $k \le \dim\Ran P_+$.

  \end{enumerate}
\end{proposition}

\begin{remark}\label{rem:spectralShift}
  The above proposition is tailored towards spectral gaps to the right of $0$, but by a spectral shift we can of course handle
  also spectral gaps to the right of any point $\gamma \in \RR$. Indeed, we have
  $\EE_{A-\gamma}((0,\infty)) = \EE_A((\gamma,\infty))$ for $\gamma \in \RR$ and analogously for $B$. Moreover, the form
  associated to the operator $B - \gamma$ is known to agree with the form $\fb - \gamma$. The latter can be seen for instance
  with an analogous reasoning as in~\cite[Proposition~10.5\,(a)]{Schm12}; cf.~also Lemma~\ref{cor:formpert} in
  Appendix~\ref{sec:heinz} below.
\end{remark}

\begin{remark}\label{rem:neg}
  Since $P_+$ and $Q_+$ are spectral projections for the respective operators, we have $\Ran(P_+Q_+|_{\fD_+}) \subset \fD_+$ and
  $\Ran(P_+Q_+|_{\cD_+}) \subset \cD_+$. In this respect, the condition $\Ran(P_+Q_+|_{\fD_+}) \supset \fD_+$ in part~(a) of
  Proposition~\ref{prop:GLS} actually means that the restriction $P_+Q_+|_{\fD_+} \colon \fD_+ \to \fD_+$ is surjective. This has
  not been formulated explicitly in the statement of \cite[Theorem~1]{GLS99} but has instead been guaranteed by the stronger
  condition
  \begin{equation*}
    \norm{ (\abs{A}+I)^{1/2} P_+Q_- (\abs{A}+I)^{-1/2} }
    <
    1
    .
  \end{equation*}
  In fact, taking into account that $\fD_+ = \Ran ((\abs{A}+I)^{-1/2}|_{\Ran P_+})$, a standard Neumann series argument in the
  Hilbert space $\Ran P_+$ then even gives bijectivity of the restriction $P_+Q_+|_{\fD_+}$, see Step~2 of the proof
  of~\cite[Theorem~1]{GLS99}. In this reasoning, the operators $(\abs{A}+I)^{\pm 1/2}$ can be replaced by
  $(\abs{A}+\alpha I)^{\pm 1/2}$ for any $\alpha>0$; if $\abs{A}$ has a bounded inverse, also $\alpha=0$ can be considered here.
  
  Of course, the above reasoning also applies in the situation of part~(b) of Proposition~\ref{prop:GLS}, but with
  $(\abs{A}+\alpha I)^{\pm 1/2}$ replaced by $(\abs{A}+\alpha I)^{\pm 1}$.
\end{remark}

In the context of our main theorems, the restriction $P_+Q_+|_{\Ran P_+}$, understood as an endomorphism of $\Ran P_+$, will
always be bijective, cf.~Remark~\ref{rem:PQbij}\,(1) below. It turns out that then the hypotheses of part~(b) in
Proposition~\ref{prop:GLS} imply those of part~(a), in which case both representations for the minimax values in
Proposition~\ref{prop:GLS} are valid. More precisely, we have the following lemma, essentially based on the well-known Heinz
inequality, cf.~Appendix~\ref{sec:heinz} below.

\begin{lemma}\label{lem:GLS}
  Assume Hypothesis~\ref{hyp:minimax} with $\Dom(A) = \Dom(B)$.
  \begin{enumerate}
    \renewcommand{\theenumi}{\alph{enumi}}
    
    \item One has $\Dom(\abs{A}^{1/2}) = \Dom(\abs{B}^{1/2})$.

    \item If $\scprod{x , Bx} \le 0$ for all $x \in \cD_-$, then $\fb[x , x] \le 0$ for all $x \in \fD_-$.

    \item If the restriction $P_+Q_+|_{\Ran P_+} \colon \Ran P_+ \to \Ran P_+$ is bijective and
          $\Ran(P_+Q_+|_{\cD_+}) \supset \cD_+$, then also $\Ran(P_+Q_+|_{\fD_+}) \supset \fD_+$.

  \end{enumerate}
\end{lemma}

\begin{proof}
  (a).
  This is a consequence of the well-known Heinz inequality, see, e.g., Corollary~\ref{cor:SchmDiss} below. Alternatively, this
  follows by classical considerations regarding operator and form boundedness, see Remark~\ref{rem:SchmDiss} below.
  
  (b).
  It follows from part~(a) that the operator $\abs{B}^{1/2} ( \abs{A}^{1/2} + I )^{-1}$ is closed and everywhere defined, hence
  bounded by the closed graph theorem. Thus,
  \begin{equation*}
    \norm{\abs{B}^{1/2}x} \le \norm{\abs{B}^{1/2}(\abs{A}^{1/2}+I)^{-1}}\cdot \norm{(\abs{A}^{1/2}+I)x}
  \end{equation*}
  for all $x\in \Dom(\abs{A}^{1/2})=\Dom(\abs{B}^{1/2})$. Since $\cD_-$ is a core for the operator
  $\abs{A|_{\Ran P_-}}^{1/2}=\abs{A}^{1/2}|_{\Ran P_-}$ with $\Dom(\abs{A}^{1/2}|_{\Ran P_-})=\fD_-$, the inequality
  $\fb[x,x]\le 0$ for $x\in\fD_-$ now follows from the hypothesis $\scprod{ x,Bx } \le 0$ for all $x\in\cD_-$ by approximation.

  (c).
  We clearly have $\Ran(P_+Q_+|_{\cD_+})=\cD_+$, $\cD_+=\Dom(A|_{\Ran P_+})$, and $\fD_+=\Dom(\abs{A|_{\Ran P_+}}^{1/2})$.
  Applying Corollary~\ref{cor:fracBij} below with the choices $\Lambda_1=\Lambda_2=A|_{\Ran P_+}$ and $S = P_+Q_+|_{\Ran P_+}$
  therefore implies that $\Ran(P_+Q_+|_{\fD_+})=\fD_+$, which proves the claim.
\end{proof}%

\begin{remark}\label{rem:PQbij}
  (1)
  In light of the identity $P_+Q_+ = P_+ - P_+Q_-$, the bijectivity of $P_+Q_+|_{\Ran P_+} \colon \Ran P_+ \to \Ran P_+$ can be
  guaranteed, for instance, by the condition $\norm{P_+Q_-} < 1$ via a standard Neumann series argument. Since
  $P_+-Q_+=P_+Q_- - P_-Q_+$ and, in particular, $\norm{P_+Q_-}\le\norm{P_+-Q_+}$, this condition holds if the stronger inequality
  $\norm{P_+-Q_+}<1$ is satisfied. In the latter case, there also is a unitary operator $U$ with $Q_+U=UP_+$, see,
  e.g.,~\cite[Theorem~I.6.32]{Kato95}, so that automatically $\dim \Ran P_+=\dim\Ran Q_+$. It is this situation we encounter in
  Theorems~\ref{thm:genSemibounded}--\ref{thm:offdiagForm}.
 
  (2)
  In the case where $B$ is an infinitesimal operator perturbation of $A$, the inequality $\norm{P_+Q_-}<1$ already implies that
  $\Ran(P_+Q_+|_{\cD_+}) \supset \cD_+$, see the following section; the particular case where $B$ is a bounded perturbation of
  $A$ has previously been considered in~\cite[Lemma~A.6]{NSTTV20}. For more general, not necessarily infinitesimal,
  perturbations, this remains so far an open problem.
\end{remark}

\section{Proof of Theorem~\ref{thm:genOpInfinitesimal}: The graph norm approach}\label{sec:graphNorm}

In this section we show that the inequality $\norm{P_+Q_-} < 1$ in the context of Theorem~\ref{thm:genOpInfinitesimal} implies
that $\Ran(P_+Q_+|_{\cD_+}) \supset \cD_+$, which is essentially what is needed to deduce Theorem~\ref{thm:genOpInfinitesimal}
from Proposition~\ref{prop:GLS} and Lemma~\ref{lem:GLS}. The main technique used to accomplish this can in fact be formulated in
a much more general framework:

Recall that for a closed operator $\Lambda$ on a Banach space with norm $\norm{\,\cdot\,}$, its domain $\Dom(\Lambda)$ can be
equipped with the~\emph{graph norm}
\begin{equation*}
 \norm{x}_\Lambda := \norm{x} + \norm{\Lambda x},\quad x\in\Dom(\Lambda),
\end{equation*}
which makes $(\Dom(\Lambda),\norm{\,\cdot\,}_\Lambda)$ a Banach space. Also recall that a linear operator $K$ with
$\Dom(K) \supset \Dom(\Lambda)$ is called~\emph{$\Lambda$-bounded with $\Lambda$-bound $\beta_* \ge 0$} if for all
$\beta > \beta_*$ there is an $\alpha \ge 0$ with
\begin{equation}\label{eq:relBound}
  \norm{Kx}
  \le
  \alpha \norm{x} + \beta \norm{\Lambda x}
  \quad\text{ for all }\
  x \in \Dom(\Lambda)
\end{equation}
and if there is no such $\alpha$ for $0 < \beta < \beta_*$.

The following lemma extends part~(a) of~\cite[Proposition~A.5]{NSTTV20}, taken from Lemma~3.9 in the author's
Ph.D.~thesis~\cite{SeelDiss}, to relatively bounded commutators.

\begin{lemma}\label{lem:specRad}
  Let $\Lambda$ be a closed operator on a Banach space, $K$ be $\Lambda$-bounded with $\Lambda$-bound $\beta_* \ge 0$, and let
  $S$ be bounded with $\Ran(S|_{\Dom(\Lambda)})\subset\Dom(\Lambda)$ and
  \begin{equation*}
    \Lambda Sx - S\Lambda x
    =
    Kx
    \quad\text{ for all }\
    x\in\Dom(\Lambda)
    .
  \end{equation*}
  Then, the restriction $S|_{\Dom(\Lambda)}$ is bounded with respect to the graph norm for $\Lambda$, and the corresponding
  spectral radius $r_\Lambda(S) = \lim_{k\to\infty} \norm{(S|_{\Dom(\Lambda)})^k}_\Lambda^{1/k}$ satisfies
  \begin{equation*}
    r_\Lambda(S)
    \leq
    \norm{S} + \beta_*
    .
  \end{equation*}
\end{lemma}

\begin{proof}
  Only small modifications to the reasoning from~\cite[Lemma~3.9]{SeelDiss},~\cite[Proposition~A.5]{NSTTV20} are necessary. For
  the sake of completeness, we reproduce the full argument here:
  
  Let $\beta > \beta_*$ and $\alpha \ge 0$ such that~\eqref{eq:relBound} holds. Then, for $x\in\Dom(\Lambda)$ one has
  \begin{equation*}
   \norm{\Lambda Sx} \le \norm{S}\norm{\Lambda x} + \norm{Kx} \le (\norm{S}+\beta)\norm{\Lambda x} + \alpha\norm{x},
  \end{equation*}
  so that
  \begin{equation*}
    \norm{Sx}_\Lambda = \norm{Sx} + \norm{\Lambda Sx} \le \bigl(\norm{S}+\beta\bigr)\norm{x}_\Lambda + \alpha\norm{x}.
  \end{equation*}
  In particular, $S|_{\Dom(\Lambda)}$ is bounded with respect to the graph norm $\norm{\,\cdot\,}_\Lambda$ with
  $\norm{S}_\Lambda\le \norm{S} + \beta + \alpha$.

  Now, a straightforward induction yields
  \begin{equation*}
    \norm{S^kx}_\Lambda
    \le
    \bigl(\norm{S} + \beta\bigr)^k\norm{x}_\Lambda + k\alpha\bigl(\norm{S}+\beta\bigr)^{k-1}\norm{x},
    \quad x\in\Dom(\Lambda),
  \end{equation*}
  for $k\in\NN$. Hence, $\norm{(S|_{\Dom(\Lambda)})^k}_\Lambda \le (\norm{S}+\beta)^k + k\alpha(\norm{S}+\beta)^{k-1}$, so that
  \begin{align*}
    r_\Lambda(S)
    &=
    \lim_{k\to\infty} \norm{(S|_{\Dom(\Lambda)})^k}_\Lambda^{1/k}
      \le
      \lim_{k\to\infty} \bigl( (\norm{S}+\beta)^k+k\alpha(\norm{S}+\beta)^{k-1} \bigr)^{1/k}\\
    &=
    \norm{S}+\beta
    .
  \end{align*}
  Since $\beta > \beta_*$ was chosen arbitrarily, this proves the claim.
\end{proof}%

We are now in position to prove Theorem~\ref{thm:genOpInfinitesimal}.

\begin{proof}[Proof of Theorem~\ref{thm:genOpInfinitesimal}]
  We mainly follow the line of reasoning in the proof of~\cite[Lemma~A.6]{NSTTV20}. Only a few additional considerations are
  necessary in order to accommodate unbounded perturbations $V$ by means of Lemma~\ref{lem:specRad}. For convenience of the
  reader, we nevertheless reproduce the whole argument here.
 
  Define $S,T \colon \Ran P_+ \to \Ran P_+$ by
  \begin{equation*}
    S := P_+Q_-|_{\Ran P_+},\quad T := P_+Q_+|_{\Ran P_+} = I_{\Ran P_+} - S.
  \end{equation*}
  By hypothesis, we have $\norm{S} \le \norm{P_+Q_-} < 1$, so that $T$ is bijective. In light of Proposition~\ref{prop:GLS} and
  Lemma~\ref{lem:GLS}, it now remains to show the inclusion $\Ran(P_+Q_+|_{\cD_+}) \supset \cD_+$, that is,
  $\Ran(T^{-1}|_{\cD_+})\subset\cD_+$. To this end, we rewrite $T^{-1}$ as a Neumann series,
  \begin{equation*}
    T^{-1}
    =
    (I_{\Ran P_+} - S)^{-1}
    =
    \sum_{k=0}^\infty S^k
    .
  \end{equation*}
  Clearly, $S$ maps the domain $\cD_+=\Dom(A|_{\Ran P_+})$ into itself, so that the inclusion $\Ran(T^{-1}|_{\cD_+})\subset\cD_+$
  holds if the above series converges also with respect to the graph norm for the closed operator $\Lambda:=A|_{\Ran P_+}$. This,
  in turn, is the case if the corresponding spectral radius $r_\Lambda(S)$ of $S$ is smaller than $1$.

  For $x\in\cD_+\subset\Ran P_+$ we compute
  \begin{align*}
    \Lambda Sx
    &=
    AP_+Q_-x
      = P_+(A+V)Q_-x - P_+VQ_-x\\
    &=
    P_+Q_-(A+V)x - P_+VQ_-x\\
    &=
    S\Lambda x + Kx
  \end{align*}
  with
  \begin{equation*}
    K := (P_+Q_-V - P_+VQ_-)|_{\Ran P_+}.
  \end{equation*}
  We show that the operator $K$ is $\Lambda$-bounded with $\Lambda$-bound $0$. Indeed, let $b > 0$, and choose $a\ge0$ with
  $\norm{Vx}\le a\norm{x}+b\norm{Ax}$ for all $x\in\Dom(A)$; recall that $V$ is infinitesimal with respect to $A$ by hypothesis.
  Then,
  \begin{equation*}
    \norm{VQ_-x}
    \le
    a\norm{Q_-x} + b\norm{AQ_-x}
    \le
    a\norm{x} + b\norm{(A+V)x} + b\norm{VQ_-x},
  \end{equation*}
  so that
  \begin{align*}
    \norm{VQ_-x}
    &\le
    \frac{a}{1-b}\norm{x} + \frac{b}{1-b}\bigl(\norm{Ax}+\norm{Vx}\bigr)\\
    &\le
    \frac{a(1+b)}{1-b}\norm{x} + \frac{b(1+b)}{1-b}\norm{Ax}.
  \end{align*}
  Thus,
  \begin{equation}\label{eq:relBoundK}
   \begin{aligned}
    \norm{Kx}
    &\le
    \norm{P_+Q_-}\norm{Vx} + \norm{VQ_-x}\\
    &\le
    a\Bigl(\norm{P_+Q_-} + \frac{1+b}{1-b}\Bigr)\norm{x} + b\Bigl(\norm{P_+Q_-}+\frac{1+b}{1-b}\Bigr)\norm{\Lambda x}
   \end{aligned}
  \end{equation}
  for $x\in\Dom(\Lambda)=\cD_+$. Since $b>0$ was chosen arbitrarily, this implies that $K$ is $\Lambda$-bounded with
  $\Lambda$-bound $0$. It therefore follows from Lemma~\ref{lem:specRad} that $r_\Lambda(S)\le\norm{S}<1$, which completes the
  proof.
\end{proof}%

\begin{remark}\label{rem:relBoundK}
  (1)
  Estimate~\eqref{eq:relBoundK} suggests that also relatively bounded perturbations $V$ that are not necessarily infinitesimal
  with respect to $A$ can be considered here. Indeed, if $b_*\in[0,1)$ is the $A$-bound of $V$, then by~\eqref{eq:relBoundK} and
  Lemma~\ref{lem:specRad} we have
  \begin{equation*}
    r_\Lambda(S)
    \le
    \norm{P_+Q_-} + b_*\Bigl(\norm{P_+Q_-}+\frac{1+b_*}{1-b_*}\Bigr)
    ,
  \end{equation*}
  and the right-hand side of the latter is smaller than $1$ if and only if
  \begin{equation*}
    \norm{P_+Q_-} < \frac{1-2b_*-b_*^2}{1-b_*^2}.
  \end{equation*}
  This is a reasonable condition on the norm $\norm{P_+Q_-}$ only for $b_*<\sqrt{2}-1$.

  (2)
  A similar result as in (1) can be obtained in terms of the $(A+V)$-bound of $V$: If for some $\tilde{b}\in[0,1)$ and
  $\tilde{a}\ge 0$ one has $\norm{Vx} \le \tilde{a}\norm{x} + \tilde{b}\norm{(A+V)x}$ for all $x\in\Dom(A)=\Dom(A+V)$, then
  standard arguments as in the above proof of Theorem~\ref{thm:genOpInfinitesimal} (see also~\cite[Lemma~2.1.6]{Tre08}) show that
  \begin{equation*}
    \norm{Vx}
    \le
    \frac{\tilde{a}}{1-\tilde{b}}\norm{x} + \frac{\tilde{b}}{1-\tilde{b}}\norm{Ax}
  \end{equation*}
  and, in turn,
  \begin{align*}
    \norm{VQ_-x}
    &\le
    \tilde{a}\norm{x} + \tilde{b}\norm{(A+V)x}
      \le
      \tilde{a}\norm{x} + \tilde{b}\norm{Ax} + \tilde{b}\norm{Vx}\\
    &\le
    \tilde{a}\Bigl( 1 + \frac{\tilde{b}}{1-\tilde{b}} \Bigr)\norm{x}
      + \tilde{b}\Bigl( 1 + \frac{\tilde{b}}{1-\tilde{b}} \Bigr)\norm{Ax}\\
    &=
    \frac{\tilde{a}}{1-\tilde{b}}\norm{x} + \frac{\tilde{b}}{1-\tilde{b}}\norm{Ax}
  \end{align*}
  for all $x\in\Dom(A)$. Plugging these into~\eqref{eq:relBoundK} gives
  \begin{align*}
    \norm{Kx}
    &\le
    \norm{P_+Q_-}\norm{Vx} + \norm{VQ_-x}\\
    &\le
    (1+\norm{P_+Q_-})\Bigl( \frac{\tilde{a}}{1-\tilde{b}}\norm{x} + \frac{\tilde{b}}{1-\tilde{b}}\norm{\Lambda x} \Bigr)
  \end{align*}
  for all $x \in \Dom(\Lambda) = \cD_+$, which eventually leads to
  \begin{equation*}
    r_\Lambda(S)
    \le
    \norm{P_+Q_-} + \frac{\tilde{b}(1+\norm{P_+Q_-})}{1-\tilde{b}}
    =
    \frac{\norm{P_+Q_-}+\tilde{b}}{1-\tilde{b}}
    .
  \end{equation*}
  The right-hand side of the latter is smaller than $1$ if and only if
  \begin{equation*}
    \norm{P_+Q_-} < 1-2\tilde{b},
  \end{equation*}
  which is a reasonable condition on $\norm{P_+Q_-}$ only for $\tilde{b} < 1/2$.
\end{remark}

\section{The block diagonalization approach}\label{sec:blockDiag}

In this section, we discuss an approach to verify the hypotheses of Proposition~\ref{prop:GLS} and Lemma~\ref{lem:GLS} which
relies on techniques previously discussed in the context of block diagonalizations of operators and forms, for instance
in~\cite{MSS16} and~\cite{GKMSV17}, respectively; cf.~also Remark~\ref{rem:MSS16} below.

Recall that for the two orthogonal projections $P_+$ and $Q_+$ from Hypothesis~\ref{hyp:minimax} the inequality
$\norm{P_+-Q_+}<1$ holds if and only if $\Ran Q_+$ can be represented as
\begin{equation}\label{eq:graph}
  \Ran Q_+ = \{ f \oplus Xf \mid f\in\Ran P_+ \}
\end{equation}
with some bounded linear operator $X\colon\Ran P_+\to\Ran P_-$; in this case, one has
\begin{equation}\label{eq:normPQX}
  \norm{P_+-Q_+}
  =
  \frac{\norm{X}}{\sqrt{1+\norm{X}{}^2}}
  ,
\end{equation}
see, e.g.,~\cite[Corollary~3.4\,(i)]{KMM03:181}. The orthogonal projection $Q_+$ can then be represented as the $2\times2$ block
operator matrices
\begin{equation}\label{eq:reprQ}
 \begin{aligned}
  Q_+
  &=
  \begin{pmatrix}
    (I_{\Ran P_+}+X^*X)^{-1} & (I_{\Ran P_+}+X^*X)^{-1}X^*\\
    X(I_{\Ran P_+}+X^*X)^{-1} & X(I_{\Ran P_+}+X^*X)^{-1}X^*
  \end{pmatrix}\\
  &=
  \begin{pmatrix}
    (I_{\Ran P_+}+X^*X)^{-1} & X^*(I_{\Ran P_-}+XX^*)^{-1}\\
    (I_{\Ran P_-}+XX^*)^{-1}X & XX^*(I_{\Ran P_-}+XX^*)^{-1}
  \end{pmatrix}
 \end{aligned}
\end{equation}
with respect to $\Ran P_+\oplus\Ran P_-$, see, e.g.,~\cite[Remark~3.6]{KMM03:181}. In particular, we have
\begin{equation}\label{eq:PQY}
  P_+Q_+|_{\Ran P_+}
  =
  (I_{\Ran P_+} + X^*X)^{-1}
  ,
\end{equation}
which is in fact the starting point for the current approach: With regard to the desired relations
$\Ran (P_+Q_+|_{\cD_+}) \supset \cD_+$ and $\Ran (P_+Q_+|_{\fD_+}) \supset \fD_+$, we need to establish that the operator
$I_{\Ran P_+} + X^*X$ maps $\cD_+$ and $\fD_+$ into $\cD_+$ and $\fD_+$, respectively.

Define the skew-symmetric operator $Y$ via the $2\times2$ block operator matrix
\begin{equation}\label{eq:defY}
  Y
  =
  \begin{pmatrix} 0 & -X^*\\ X & 0 \end{pmatrix}
\end{equation}
with respect to $\Ran P_+\oplus\Ran P_-$. Then, the operators $I\pm Y$ are bijective with
\begin{equation}\label{eq:Ypm}
  (I-Y)(I+Y)
  =
  \begin{pmatrix}
    I_{\Ran P_+}+X^*X & 0\\
    0 & I_{\Ran P_-}+XX^*
  \end{pmatrix}
  .
\end{equation}

The following lemma is extracted from various sources. We comment on this afterwards in Remark~\ref{rem:PQXY} below.
\begin{lemma}\label{lem:PQXY}
  Suppose that the projections $P_+$ and $Q_+$ from Hypothesis~\ref{hyp:minimax} satisfy $\norm{P_+ - Q_+} < 1$, and let the
  operators $X$ and $Y$ be as in~\eqref{eq:graph} and~\eqref{eq:defY}, respectively. Moreover, let $\cC$ be an invariant subspace
  for both $P_+$ and $Q_+$ such that $\cC = (\cC \cap \Ran P_+) \oplus (\cC \cap \Ran P_-) =: \cC_+ \oplus \cC_-$.
  
  Then, the following are equivalent:
  \begin{enumerate}
    \renewcommand{\theenumi}{\roman{enumi}}

    \item $I_{\Ran P_+} + X^*X$ maps $\cC_+$ into itself;

    \item $I_{\Ran P_-} + XX^*$ maps $\cC_-$ into itself;

    \item $Y$ maps $\cC$ into itself;

    \item $(I+Y)$ maps $\cC$ into itself;

    \item $(I-Y)$ maps $\cC$ into itself.

  \end{enumerate}
\end{lemma}

\begin{proof}
  Clearly, the hypotheses imply that $P_+Q_+$ maps $\cC$ into $\cC_+$ and $P_-Q_+$ maps $\cC$ into $\cC_-$.
  
  (i)$\Rightarrow$(ii).
  Let $g \in \cC_-$. Using the first representation in~\eqref{eq:reprQ}, we then have
  $(I_{\Ran P_+} + X^*X)^{-1}X^*g = (P_+Q_+|_{\Ran P_-})g \in \cC_+$. Hence, $X^*g \in \cC_+$ by~(i) and, in turn,
  $h := (I_{\Ran P_+} + X^*X)X^*g \in \cC_+$. Using again~\eqref{eq:reprQ}, this yields
  \begin{align*}
    (I_{\Ran P_-} + XX^*)g
    &=
    g + XX^*g
      =
      g + X(I_{\Ran P_+}+X^*X)^{-1}h\\
    &=
    g + (P_-Q_+|_{\Ran P_+})h \in \cC_-
    .
  \end{align*}
  As a byproduct, we have also shown that $X^*$ maps $\cC_-$ into $\cC_+$.

  (ii)$\Rightarrow$(i).
  Using the identities $(I_{\Ran P_-} + XX^*)^{-1}X = P_-Q_+|_{\Ran P_+}$ and $X^*(I_{\Ran P_-}+XX^*)^{-1} = P_+Q_+|_{\Ran P_-}$
  taken from the second representation in~\eqref{eq:reprQ}, the proof is completely analogous to the implication
  (i)$\Rightarrow$(ii). In particular, we likewise obtain as a byproduct that $X$ maps $\cC_+$ into $\cC_-$.

  (i),(ii)$\Rightarrow$(iii).
  We have already seen that $X$ maps $\cC_+$ into $\cC_-$ and that $X^*$ maps $\cC_-$ into $\cC_+$. Taking into account that
  $\cC = \cC_+ \oplus \cC_-$, this means that $Y$ maps $\cC$ into itself.

  (iii)$\Leftrightarrow$(iv),(v).
  This is clear.

  (iv),(v)$\Rightarrow$(i),(ii).
  This follows immediately from identity~\eqref{eq:Ypm}.
\end{proof}%

\begin{remark}\label{rem:PQXY}
  The proof of the equivalence (i)$\Leftrightarrow$(ii) and the one of the implication (i),(ii)$\Rightarrow$(iii) in
  Lemma~\ref{lem:PQXY} are extracted from the proof of~\cite[Theorem~5.1]{GKMSV17}; see also~\cite[Theorem~6.3.1 and
  Lemma~6.3.3]{SchmDiss}.

  The equivalence (iv)$\Leftrightarrow$(v) can alternatively be directly obtained from the identity
  \begin{equation*}
    \begin{pmatrix}
      I_{\Ran P_+} & 0\\
      0 & -I_{\Ran P_-}
    \end{pmatrix}
    (I + Y)
    \begin{pmatrix}
      I_{\Ran P_+} & 0\\
      0 & -I_{\Ran P_-}
    \end{pmatrix}
    =
    I - Y
    .
  \end{equation*}
  Such an argument has been used in the proof of~\cite[Proposition~3.3]{MSS16}.

  The implication (iv),(v)$\Rightarrow$(i) can essentially be found in the proof of~\cite[Theorem~5.1]{GKMSV17}
  and~\cite[Remark~6.3.2]{SchmDiss}.
\end{remark}

Below, we apply Lemma~\ref{lem:PQXY} with $\cC = \Dom(A) = \Dom(B) = \cD_+ \oplus \cD_-$\linebreak or
$\cC = \Dom(\abs{A}^{1/2}) = \Dom(\abs{B}^{1/2}) = \fD_+ \oplus \fD_-$, depending on the situation. The easiest case is
encountered in Theorem~\ref{thm:genSemibounded}, where we are in the semibounded setting:

\begin{proof}[Proof of Theorem~\ref{thm:genSemibounded}]
  Let $\Dom(\abs{A}^{1/2}) = \Dom(\abs{B}^{1/2})$ and $\fb[x , x] \le 0$ for all $x \in \fD_-$. We then have $\fD_- = \Ran P_-$
  if $A$ is bounded from below and $\fD_+ = \Ran P_+$ if $A$ is bounded from above. Hence, item~(ii) or (i), respectively, in
  Lemma~\ref{lem:PQXY} with $\cC = \fD_+ \oplus \fD_-$ is automatically satisfied. In any case, we have by Lemma~\ref{lem:PQXY}
  that $I_{\Ran P_+} + X^*X$ maps $\fD_+$ into $\fD_+$, which by identity~\eqref{eq:PQY} means that
  $\Ran (P_+Q_+|_{\fD_+}) \supset \fD_+$. The representation~\eqref{eq:genSemibounded:form} now follows from
  Proposition~\ref{prop:GLS}\,(a) and Remark~\ref{rem:PQbij}\,(1). If even $\Dom(A) = \Dom(B)$ and $\scprod{ x , Bx } \le 0$ for
  all $x \in \cD_-$, we use the same reasoning as above with $\fD_+$ and $\fD_-$ replaced by $\cD_+$ and $\cD_-$, respectively,
  and obtain representation~\eqref{eq:genSemibounded:op} from Proposition~\ref{prop:GLS}\,(b) and Remark~\ref{rem:PQbij}\,(1).
  The representation~\eqref{eq:genSemibounded:form} is then still valid by Lemma~\ref{lem:GLS} and the first part of the proof.
\end{proof}%

While certain conditions for Proposition~\ref{prop:GLS} and Lemma~\ref{lem:GLS} are part of the hypotheses of
Theorems~\ref{thm:genOpInfinitesimal} and~\ref{thm:genSemibounded}, in the situations of Theorems~\ref{thm:offdiagOp}
and~\ref{thm:offdiagForm} these need to be verified explicitly from the specific hypotheses at hand. Here, we rely on previous
considerations on block diagonalizations for block operator matrices and forms. In case of Theorem~\ref{thm:offdiagOp}, the
crucial ingredient is presented in the following result, extracted from~\cite{MSS16}. An earlier result in this direction is
commented on in Remark~\ref{rem:MSS16}\,(2) below.

\begin{proposition}[see~{\cite[Theorem~6.1]{MSS16}}]\label{prop:MSS16}
  In the situation of Theorem~\ref{thm:offdiagOp} one has $\norm{P_+-Q_+} \le \sqrt{2}/2 <1$, and the operator identity
  \begin{equation}\label{eq:blockDiag}
    (I-Y)(A+V)(I-Y)^{-1} = A-YV
  \end{equation}
  holds with $Y$ as in~\eqref{eq:defY}.
\end{proposition}

\begin{proof}
  Set $V_0 := V|_{\Dom(A)}$, so that we have $B = A+V = A+V_0$ as well as $A-YV = A-YV_0$. Clearly, the hypotheses on $V$ ensure
  that $V_0$ is $A$-bounded with $A$-bound $b_*<1$ and off-diagonal with respect to the decomposition $\Ran P_+\oplus\Ran P_-$.
  By~\cite[Lemma~6.3]{MSS16} we now have 
  \begin{equation*}
    \Ker(A+V_0)
    \subset
    \Ker A
    \subset
    \Ran P_-
    .
  \end{equation*}
  In light of~\eqref{eq:normPQX}, the claim therefore is just an instance of~\cite[Theorem~6.1]{MSS16}.
\end{proof}%

\begin{remark}\label{rem:MSS16}
  (1)
  Let $A_\pm:=A|_{\Ran P_\pm}$ be the parts of $A$ associated with the subspaces $\Ran P_\pm$, and write
  \begin{equation*}
    V|_{\Dom(A)}
    =
    \begin{pmatrix} 0 & W\\ W^* & 0 \end{pmatrix}
    ,
  \end{equation*}
  where $W\colon \Ran P_-\supset\cD_-\to\Ran P_+$ is given by $Wx:=P_+Vx$, $x\in\cD_-$. Then,
  \begin{equation*}
    A - YV
    =
    \begin{pmatrix}
      A_+ - X^*W^* & 0\\
      0 & A_- + XW
    \end{pmatrix}
    .
  \end{equation*} 
  In this sense, identity~\eqref{eq:blockDiag} can be viewed as a block diagonalization of the operator $A+V$. For a more
  detailed discussion of block diagonalizations and operator Riccati equations in the operator setting, the reader is referred
  to~\cite{MSS16} and the references cited therein.

  (2)
  In the particular case where $0$ belongs to the resolvent set of $A$, the conclusion of Proposition~\ref{prop:MSS16} can be
  inferred also from~\cite[Theorems~2.7.21 and~2.8.5]{Tre08}.
\end{remark}

\begin{proof}[Proof of Theorem~\ref{thm:offdiagOp}]
  For $x \in \cD_-$, we have
  \begin{equation*}
    \scprod{ x , Vx }
    =
    \scprod{ P_-x , VP_-x }
    =
    \scprod{ x , P_-VP_-x }
    =
    0
  \end{equation*}
  and, thus,
  \begin{equation*}
    \scprod{ x , (A+V)x }
    =
    \scprod{ x , Ax }
    \le
    0
    .
  \end{equation*}
  Moreover, by Proposition~\ref{prop:MSS16} the inequality $\norm{P_+-Q_+}<1$ is satisfied. Let $Y$ be as in~\eqref{eq:defY}.
  Since $\Dom(A+V)=\Dom(A)=\Dom(A-YV)$, it then follows from identity~\eqref{eq:blockDiag} that $I-Y$ maps
  $\cC := \Dom(A) = \cD_+ \oplus \cD_-$ into itself. In turn, Lemma~\ref{lem:PQXY} implies that $I_{\Ran P_+} + X^*X$ maps
  $\cD_+$ into itself, which by identity~\eqref{eq:PQY} means that $\Ran(P_+Q_+|_{\cD_+}) \supset \cD_+$. The claim now follows
  from Proposition~\ref{prop:GLS}, Lemma~\ref{lem:GLS}, and Remark~\ref{rem:PQbij}\,(1).
\end{proof}

To the best of the author's knowledge, no direct analogue of Proposition~\ref{prop:MSS16} is known so far in the setting of form
rather than operator perturbations. Although the inequality $\norm{P_+-Q_+} \le \sqrt{2}/2$ can be established here as well under
fairly reasonable assumptions, see~\cite[Theorem~3.3]{GKMSV17}, the mapping properties of the operators $I \pm Y$ connected with
a corresponding diagonalization related to~\eqref{eq:blockDiag} are much harder to verify. The situation is even more subtle
there since also the domain equality $\Dom(\abs{A}^{1/2}) = \Dom(\abs{B}^{1/2})$ needs careful treatment. The latter is
conjectured to hold in a general off-diagonal form perturbation framework~\cite[Remark~2.7]{GKMV13}. Some characterizations have
been discussed in~\cite[Theorem~3.8]{Schm15}, but they all are hard to verify in a general abstract setting. A compromise in this
direction is to require that the form $\fb$ is semibounded, see~\cite[Lemma~3.9]{Schm15} and~\cite[Lemma~2.7]{GKMSV17}, which
forces the diagonal form $\fa$ to be semibounded as well, see below. As in the situation of Theorem~\ref{thm:genSemibounded}
above, this simplifies matters immensely:

\begin{proof}[Proof of Theorem~\ref{thm:offdiagForm}]
  For $x \in \fD_- = \Ran P_- \cap \Dom[\fa]$ we have
  \begin{equation*}
    \fv[ P_-x , P_-x ]
    =
    0
  \end{equation*}
  and, thus,
  \begin{equation*}
    \fb[ x , x ]
    =
    \fa[ x , x ]
    \le
    0
    .
  \end{equation*}
  In the same way, we see that $\fb[ x , x ] = \fa[ x , x ]$ for $x \in \fD_+$, which by the identity
  $\fa[ x , x ] = \fa[ P_+x , P_+x ] + \fa[ P_-x , P_-x ]$ for all $x \in \Dom[\fa]$ implies that along with $\fb$ also the form
  $\fa$ is semibounded; cf.~the proof of~\cite[Lemma~2.7]{GKMSV17}. Since also
  $\Dom(\abs{B}^{1/2}) = \Dom[\fb] = \Dom[\fa] = \Dom(\abs{A}^{1/2})$ by hypothesis and in view of
  Theorem~\ref{thm:genSemibounded}, it only remains to show that $\norm{ P_+ - Q_+ } < 1$.

  To this end, we first show that $\fb$ is a semibounded~\emph{saddle-point form} in the sense of~\cite[Section~2]{GKMSV17}: Let
  $m \in \RR$ be the lower (resp.~upper) bound of $\fa$. We then have
  \begin{equation*}
    \abs{(\fa-m)[ x , x ]}
    =
    \norm{ \abs{A-m}^{1/2}x }^2
    \le
    \norm{ \abs{A-m}^{1/2}(\abs{A}^{1/2}+I)^{-1} } \norm{ (\abs{A}^{1/2}+I)x }^2
  \end{equation*}
  for all $x \in \Dom[\fa]$, where $\abs{A-m}^{1/2}(\abs{A}^{1/2}+I)^{-1}$ is closed and everywhere defined, hence bounded by the
  closed graph theorem. From this and the hypothesis on $\fv$ we see that
  \begin{equation*}
    \abs{ \fv[ x , x ] }
    \le
    \beta\bigl( \norm{\abs{A}^{1/2}x}^2 + \norm{x}^2 \bigr)
    \quad\text{ for all }\
    x \in \Dom[\fa]
  \end{equation*}
  with some $\beta \ge 0$, which means that $\fb = \fa + \fv = \fa + \fv_0$ with $\fv_0 := \fv|_{\Dom[\fa]}$ is indeed a
  semibounded saddle-point form.

  It now follows from~\cite[Theorem~2.13]{Schm15} that
  \begin{equation*}
    \Ker B
    \subset
    \Ker A
    \subset
    \Ran P_-
    .
  \end{equation*}
  In turn,~\cite[Theorem~3.3]{GKMSV17} and~\eqref{eq:normPQX} give $\norm{ P_+ - Q_+ } \le \sqrt{2}/2 < 1$, which completes the
  proof.
\end{proof}%

\appendix

\section{Proof of Proposition~\ref{prop:GLS}}\label{sec:GLS}

The following abstract minimax principle is extracted from \cite[Theorem~1]{GLS99} and its proof follows the one in \cite{GLS99}
almost literally. However, the result below is formulated in such a way that the operator and form settings can be handled
simultaneously. To this end, we use a suitable subset $\cC$ of $\Dom(\abs{B}^{1/2})$ satisfying $\Dom(B) \subset \cC$. Part~(a)
of Proposition~\ref{prop:GLS} then agrees with the choice $\cC = \Dom(\abs{B}^{1/2})$, whereas Part~(b) of
Proposition~\ref{prop:GLS} corresponds to $\cC = \Dom(B)$.

\begin{proposition}[cf.~{\cite[Theorem~1]{GLS99}}]
	Assume Hypothesis~\ref{hyp:minimax}. Moreover, let $\cC$ be a subspace satisfying
	$\Dom(B) \subset \cC \subset \Dom(\abs{B}^{1/2})$ such that $\cC$ is invariant for $P_+$ and $Q_+$. Set
	$\cC_\pm := \cC \cap \Ran P_\pm$ and suppose that $\fb[ x , x ] \le 0$ for all $x \in \cC_-$.
	\begin{enumerate}
	  \renewcommand{\theenumi}{\alph{enumi}}

    \item
    For all $k \in \NN$ with $k \le \dim \Ran Q_+$ we have $k \le \dim \Ran P_+$ and
    \begin{equation*}
      \inf_{\substack{\fM_+ \subset \cC_+\\ \dim\fM_+ = k}} \sup_{\substack{x \in \fM_+ \oplus \cC_-\\ \norm{x}=1}} \fb[ x , x ]
      \le
      \lambda_k(B|_{\Ran Q_+})
      .
    \end{equation*}

    \item
    If, in addition, $\Ran(P_+Q_+|_{\cC_+}) \supset \cC_+$ holds, then for all $k \in \NN$ with $k \le \dim \Ran P_+$ we have
    $k \le \dim \Ran Q_+$ and
    \begin{equation*}
      \inf_{\substack{\fM_+ \subset \cC_+\\ \dim\fM_+ = k}} \sup_{\substack{x \in \fM_+ \oplus \cC_-\\ \norm{x}=1}} \fb[ x , x ]
      =
      \lambda_k(B|_{\Ran Q_+})
      .
    \end{equation*}

	\end{enumerate}
\end{proposition}

\begin{proof}
  We first show that the restriction
  \begin{equation}\label{eq:restriction}
    P_+|_{\cC \cap \Ran Q_+} \colon \cC \cap \Ran Q_+ \to \cC_+
  \end{equation}
  is injective. Indeed, assume to the contrary that $P_+x = 0$ for some non-zero $x \in \cC \cap \Ran Q_+$. Then, on the one hand
  we have $\fb[x , x] > 0$, and on the other hand $x \in \Ran P_-$, that is, $x \in \cC_-$. The latter gives $\fb[x , x] \le 0$
  by hypothesis, a contradiction.

  (a).
  Let $k \in \NN$ with $k \le \dim \Ran Q_+$. Let $\eps > 0$ be arbitrary, and abbreviate $\lambda_k = \lambda_k(B|_{\Ran Q_+})$.
  Consider the subspace
  \begin{equation*}
    \fM
    :=
    \Ran\EE_B((0,\lambda_k+\eps))
    \subset
    \cC \cap \Ran Q_+
    ,
  \end{equation*}
  and denote by $P_\fM$ the orthogonal projection onto $\fM$. We clearly have $\dim\fM \ge k$. Choose any subspace
  $\fN \subset \fM$ with $\dim\fN = k$. The injectivity of \eqref{eq:restriction} then gives
  $\dim \Ran P_+ \ge \dim \Ran P_+|_\fN = \dim \fN = k$.

  Let $x = P_+ u \oplus w$, $\norm{x} = 1$, with $u \in \fN$ and $w \in \cC_-$. Then, $x = u + v$ with $v = w - P_-u \in \cC_-$.
  Moreover, $P_\fM v \in \fM$ and, taking into account that $\fM$ is reducing for $B$, we have
  \begin{equation}\label{eq:bnonneg}
    \fb[ v , P_\fM v ]
    =
    \fb[ P_\fM v , v ]
    =
    \fb[ P_\fM v , P_\fM v ]
    \ge
    0
    .
  \end{equation}
  Set $x_1 := u + P_\fM v \in \fM$ and $x_2 := v - P_\fM v \in (\cC_- + \fM) \cap \fM^\perp$. From \eqref{eq:bnonneg} it follows
  that
  \begin{equation*}
    \fb[ x_2 , x_2 ]
    =
    \fb[ v , v ] - \fb[ P_\fM v , P_\fM v ]
    \le
    \fb[ v , v ]
    \le
    0
    .
  \end{equation*}
  In turn, taking into account that $x = u + v = x_1 \oplus x_2$ and that $\fM$ and $\fM^\perp$ are reducing for $B$, we obtain
  \begin{equation*}
    \fb[ x , x ]
    =
    \fb[ x_1 , x_1 ] + \fb[ x_2 , x_2 ]
    \le
    \fb[ x_1 , x_1 ]
    <
    (\lambda_k + \eps)\norm{x_1}^2
    \le
    \lambda_k + \eps
    .
  \end{equation*}
  In light of $\dim \Ran P_+|_\fN = k$ as observed above, we thus conclude that
  \begin{equation*}
    \inf_{\substack{\fM_+ \subset \cC_+\\ \dim\fM_+ = k}} \sup_{\substack{x \in \fM_+ \oplus \cC_-\\ \norm{x} = 1}} \fb[ x , x ]
    \le
    \sup_{\substack{x \in \Ran P_+|_\fN \oplus \cC_-\\ \norm{x} = 1}} \fb[ x , x ]
    \le
    \lambda_k + \eps
    .
  \end{equation*}
  Since $\eps$ was chosen arbitrarily, this proves (a).

  (b).
  The additional assumption $\Ran(P_+Q_+|_{\cC_+}) \supset \cC_+$ guarantees that the restriction \eqref{eq:restriction} is also
  surjective, hence bijective. Let $k \in \NN$ with $k \le \dim \Ran P_+$, and let $\fM_+$ be any subspace of $\cC_+$ with
  $\dim\fM_+ = k$; note that $\cC_+$ is dense in $\Ran P_+$, so that such a subspace indeed exists. By bijectivity of
  \eqref{eq:restriction}, there is a subspace $\fM \subset \cC \cap \Ran Q_+$ with $\dim\fM = k$ and $\fM_+ = \Ran P_+|_\fM$. In
  particular, we have $k \le \dim \Ran Q_+$. Since $\fM \subset \cC$, we have $\fM \subset \fM_+ \oplus \cC_-$ and, therefore,
  \begin{equation*}
    \sup_{\substack{x \in \fM_+ \oplus \cC_-\\ \norm{x}=1}} \fb[ x , x ]
    \ge
    \sup_{\substack{x \in \fM\\ \norm{x}=1}} \fb[ x , x ]
    \ge
    \inf_{\substack{\fN \subset \cC \cap \Ran Q_+\\ \dim\fN = k}} \sup_{\substack{x \in \fN\\ \norm{x}=1}} \fb[ x , x ]
    \ge
    \lambda_k(B|_{\Ran Q_+})
    .
  \end{equation*}
  Since $\fM_+$ was chosen arbitrarily, together with part (a) this proves (b) and, thus, completes the proof.
\end{proof}%

\section{Heinz inequality}\label{sec:heinz}

In this appendix we discuss some consequences of the well-known Heinz inequality. These consequences or particular cases thereof
are used at various spots of the main part of the paper, but they may also be of independent interest. They are probably
folklore, but in lack of a suitable reference they are nevertheless presented here in full detail.

Throughout this appendix, we denote the norm associated with the inner product of a Hilbert space $\cH$ by
$\norm{\,\cdot\,}_\cH$.

The following variant of the Heinz inequality is taken from~\cite{Kre71}.

\begin{proposition}[{\cite[Theorem~I.7.1]{Kre71}}]\label{propHeinz}
  Let $\Lambda_1$ and $\Lambda_2$ be strictly positive self-adjoint operators on Hilbert spaces $\cH_1$ and $\cH_2$,
  respectively. Moreover, let $S\colon\cH_1\to\cH_2$ be a bounded operator mapping $\Dom(\Lambda_1)$ into $\Dom(\Lambda_2)$, and
  suppose that there is a constant $C\ge0$ such that
  \begin{equation*}
    \norm{\Lambda_2Sx}_{\cH_2}
    \le
    C\cdot\norm{\Lambda_1x}_{\cH_1}
    \quad\text{ for all }\
    x\in\Dom(\Lambda_1)
    .
  \end{equation*}
  Then, for all $\nu\in[0,1]$, the operator $S$ maps $\Dom(\Lambda_1^\nu)$ into $\Dom(\Lambda_2^\nu)$, and for all
  $x\in\Dom(\Lambda_1^\nu)$ one has
  \begin{equation*}
    \norm{\Lambda_2^\nu Sx}_{\cH_2}
    \le
    C^\nu\norm{S}_{\cH_1\to\cH_2}^{1-\nu}\norm{\Lambda_1^\nu x}_{\cH_1}
    .
  \end{equation*}
\end{proposition}

The above result admits the following extension to closed densely defined operators between Hilbert spaces. For a generalization
of Proposition~\ref{propHeinz} to maximal accretive operators, see~\cite{Kato61}.

\begin{proposition}\label{prop:heinz}
  Let $\cH_1$, $\cH_2$, $\cK_1$, and $\cK_2$ be Hilbert spaces, and let $\Lambda_1\colon\cH_1\supset\Dom(\Lambda_1)\to\cK_1$ and
  $\Lambda_2\colon\cH_2\supset\Dom(\Lambda_2)\to\cK_2$ be closed densely defined operators. Moreover, let $S\colon\cH_1\to\cH_2$
  be a bounded operator mapping $\Dom(\Lambda_1)$ into $\Dom(\Lambda_2)$, and suppose that there is a constant $C\ge0$ such that
  \begin{equation*}
    \norm{\Lambda_2Sx}_{\cK_2}
    \le
    C\cdot\norm{\Lambda_1x}_{\cK_1}
    \quad\text{ for all }\
    x\in\Dom(\Lambda_1)
    .
  \end{equation*}
  Then, for all $\nu\in[0,1]$, the operator $S$ maps $\Dom(\abs{\Lambda_1}^\nu)$ into $\Dom(\abs{\Lambda_2}^\nu)$.
\end{proposition}

\begin{proof}
  Recall that $\norm{\Lambda_jy}_{\cK_j}=\norm{\abs{\Lambda_j}y}_{\cH_j}$ for all $y\in\Dom(\Lambda_j)=\Dom(\abs{\Lambda_j})$,
  $j=1,2$. Moreover, the operator $S$ maps $\Dom(\abs{\Lambda_1}+I_{\cH_1})=\Dom(\Lambda_1)$ into
  $\Dom(\abs{\Lambda_2}+I_{\cH_2})=\Dom(\Lambda_2)$ by hypothesis. We estimate
  \begin{align*}
    \norm{(\abs{\Lambda_2}+I_{\cH_2})Sx}_{\cH_2}
    &\le
    \norm{\Lambda_2Sx}_{\cK_2} + \norm{Sx}_{\cH_2} \le C \norm{\Lambda_1x}_{\cK_1} + \norm{Sx}_{\cH_2}\\
    &\le
    \widetilde{C} \norm{(\abs{\Lambda_1}+I_{\cH_1})x}_{\cH_1}
  \end{align*}
  for all $x\in\Dom(\Lambda_1)$ with
  \begin{equation*}
    \widetilde{C}
    :=
    C\norm{\Lambda_1(\abs{\Lambda_1}+I_{\cH_1})^{-1}}_{\cH_1\to\cK_1} + \norm{S(\abs{\Lambda_1}+I_{\cH_1})^{-1}}_{\cH_1\to\cH_2}
    .
  \end{equation*}
  Here, we have taken into account that $\Lambda_1(\abs{\Lambda_1}+I_{\cH_1})^{-1}$ is a closed and everywhere defined operator
  from $\cH_1$ to $\cK_1$, hence bounded by the closed graph theorem.

  Applying Proposition~\ref{propHeinz} now yields that $S$ maps $\Dom((\abs{\Lambda_1}+I_{\cH_1})^\nu)$ into
  $\Dom((\abs{\Lambda_2}+I_{\cH_2})^\nu)$ for all $\nu\in[0,1]$. It remains to observe that by functional calculus one has
  $\Dom((\abs{\Lambda_j}+I_{\cH_j})^\nu)=\Dom(\abs{\Lambda_j}^\nu)$ for $j\in\{1,2\}$, which completes the proof.
\end{proof}%

We now obtain several easy corollaries.

\begin{corollary}[{cf.~\cite[Corollary~3.3]{Schm15}}]\label{cor:SchmDiss}
  Let $\cH$, $\cK_1$, and $\cK_2$ be Hilbert spaces, and let $\Lambda_1\colon\cH\supset\Dom(\Lambda_1)\to\cK_1$ and
  $\Lambda_2\colon\cH\supset\Dom(\Lambda_2)\to\cK_2$ be closed densely defined operators.

  If $\Dom(\Lambda_1) \subset \Dom(\Lambda_2)$, then $\Dom(\abs{\Lambda_1}^\nu) \subset \Dom(\abs{\Lambda_2}^\nu)$ for all
  $\nu \in [0,1]$. Moreover, if $\Dom(\Lambda_1) = \Dom(\Lambda_2)$, then also
  $\Dom(\abs{\Lambda_1}^\nu) = \Dom(\abs{\Lambda_2}^\nu)$ for all $\nu \in [0,1]$.
\end{corollary}

\begin{proof}
  Suppose that $\Dom(\Lambda_1) \subset \Dom(\Lambda_2)$. Since $\Dom(\abs{\Lambda_1}) = \Dom(\Lambda_1)$, we have as in the
  proof of the preceding proposition that $\Lambda_2(\abs{\Lambda_1}+I_\cH)^{-1}$ is a closed everywhere defined, hence bounded,
  operator from $\cH$ to $\cK_2$. Thus,
  \begin{equation*}
    \norm{\Lambda_2x}_{\cK_2}
    \le
    \norm{\Lambda_2(\abs{\Lambda_1}+I_\cH)^{-1}}_{\cH\to\cK_2} \cdot \norm{(\abs{\Lambda_1}+I_\cH)x}_{\cH}
  \end{equation*}
  for all $x\in\Dom(\Lambda_1)$, and applying Proposition~\ref{prop:heinz} with $S=I_\cH$ yields that
  \begin{equation*}
    \Dom(\abs{\Lambda_1}^\nu)
    =
    \Dom((\abs{\Lambda_1}+I_\cH)^\nu)\subset\Dom(\abs{\Lambda_2}^\nu)
    .
  \end{equation*}
  If also $\Dom(\Lambda_1) \supset \Dom(\Lambda_2)$, the above with switched roles of $\Lambda_1$ and $\Lambda_2$ yields that
  also $\Dom(\abs{\Lambda_2}^\nu) \subset \Dom(\abs{\Lambda_1}^\nu)$, which completes the proof.
\end{proof}%

\begin{remark}\label{rem:SchmDiss}
  In the particular case of $\nu = 1/2$, Corollary~\ref{cor:SchmDiss} can alternatively be proved with classical considerations
  regarding operator and form boundedness:

  If $\Dom(\Lambda_1) \subset \Dom(\Lambda_2)$, then also $\Dom(\abs{\Lambda_1}) \subset \Dom(\abs{\Lambda_2})$, so that
  $\abs{\Lambda_2}$ is relatively operator bounded with respect to $\abs{\Lambda_1}$, see, e.g.,~\cite[Remark~IV.1.5]{Kato95}. In
  turn, by~\cite[Theorem~VI.1.38]{Kato95}, $\abs{\Lambda_2}$ is also form bounded with respect to $\abs{\Lambda_1}$, which
  extends to the closure of the forms. The latter includes that
  $\Dom(\abs{\Lambda_1}^{1/2}) \subset \Dom(\abs{\Lambda_2}^{1/2})$.
\end{remark}

\begin{corollary}\label{cor:fracBij}
  Let $\Lambda_1$ and $\Lambda_2$ be as in Proposition~\ref{prop:heinz}, and suppose that $S\colon\cH_1\to\cH_2$ is bounded and
  bijective with $\Ran(S|_{\Dom(\Lambda_1)})=\Dom(\Lambda_2)$. Then, one has
  $\Ran(S|_{\Dom(\abs{\Lambda_1}^\nu)})=\Dom(\abs{\Lambda_2}^\nu)$ for all $\nu\in[0,1]$.
\end{corollary}

\begin{proof}
  Consider the closed densely defined operator $\Lambda_3 := \Lambda_1S^{-1}$ with domain
  $\Dom(\Lambda_3) = \Ran(S|_{\Dom(\Lambda_1)}) = \Dom(\Lambda_2)$. By definition, $S$ maps $\Dom(\Lambda_1)$ onto
  $\Dom(\Lambda_3)$. Moreover, we have $\Lambda_3 Sx = \Lambda_1 x$ and, in particular,
  \begin{equation*}
    \norm{ \Lambda_3 Sx }_{\cK_1} = \norm{\Lambda_1x}_{\cK_1}
  \end{equation*}
  for all $x\in\Dom(\Lambda_1)$. Proposition~\ref{prop:heinz} now implies that $S$ maps $\Dom(\abs{\Lambda_1}^\nu)$ into
  $\Dom(\abs{\Lambda_3}^\nu)$. Since $\Dom(\abs{\Lambda_3}^\nu)=\Dom(\abs{\Lambda_2}^\nu)$ for all $\nu \in [0,1]$ in light of
  Corollary~\ref{cor:SchmDiss}, this proves the inclusion $\Ran(S|_{\Dom(\abs{\Lambda_1}^\nu)})\subset\Dom(\abs{\Lambda_2}^\nu)$.

  Since $S$ is bijective and $S^{-1}$ maps $\Dom(\Lambda_2)$ onto $\Dom(\Lambda_1)$ by hypothesis, one verifies in an analogous
  way that $S^{-1}$ maps $\Dom(\abs{\Lambda_2}^\nu)$ into $\Dom(\abs{\Lambda_1}^\nu)$. This shows the converse inclusion and,
  hence, completes the proof.
\end{proof}%

The last corollary discussed here is related to the question whether an operator sum represents the sum of the corresponding
forms. Part~(b) of this corollary can in some sense also be regarded as an extension of~\cite[Lemma~2.2.7]{SchmDiss} to not
necessarily off-diagonal perturbations.

\begin{corollary}[{cf.~\cite[Lemma~2.2.7]{SchmDiss}}]\label{cor:formpert}
  Let $\Lambda$ be a self-adjoint operator on a Hilbert space with inner product $\scprod{\cdot , \cdot}$, and let $K$ be an
  operator on the same Hilbert space.

  \begin{enumerate}
    \renewcommand{\theenumi}{\alph{enumi}}

    \item
    If $K$ is symmetric with $\Dom(K) \supset \Dom(\abs{\Lambda}^{1/2})$, then the operator sum $\Lambda + K$ defines a
    self-adjoint operator with
    \begin{equation*}
      \scprod{ \abs{ \Lambda + K }^{1/2}x , \sign( \Lambda + K )\abs{ \Lambda + K}^{1/2}y }
      =
      \scprod{ \abs{\Lambda}^{1/2}x , \sign(\Lambda)\abs{\Lambda}^{1/2}y } + \scprod{ x, Ky }
    \end{equation*}
    for all $x,y \in \Dom(\abs{\Lambda}^{1/2}) = \Dom(\abs{\Lambda+K}^{1/2})$.

    \item
    If $K$ is self-adjoint with $\Dom(K) \supset \Dom(\Lambda)$ such that $\Lambda + K$ is self-adjoint on
    $\Dom(\Lambda+K) = \Dom(\Lambda)$, then
    \begin{multline*}
      \scprod{ \abs{ \Lambda + K }^{1/2}x , \sign( \Lambda + K )\abs{ \Lambda + K }^{1/2}y }\\
      =
      \scprod{ \abs{\Lambda}^{1/2}x , \sign(\Lambda)\abs{\Lambda}^{1/2}y } +
        \scprod{ \abs{K}^{1/2}x, \sign(K)\abs{K}^{1/2}y }
    \end{multline*}
    for all $x,y \in \Dom(\abs{\Lambda}^{1/2}) = \Dom(\abs{\Lambda+K}^{1/2})$.

  \end{enumerate}
\end{corollary}

\begin{proof}
  (a).
  For $x,y \in \Dom(\Lambda) = \Dom(\Lambda+K)$, both sides of the claimed identity clearly agree. For general
  $x,y \in \Dom(\abs{\Lambda}^{1/2}) = \Dom(\abs{\Lambda+K}^{1/2})$, we use that $\Dom(\Lambda)$ is an operator core for
  $\abs{\Lambda}^{1/2}$ and approximate both $x$ and $y$ within $\Dom(\Lambda)$ with respect to the graph norm of
  $\abs{\Lambda}^{1/2}$. In order to conclude the claim, it only remains to show that $x$ and $y$ are then also approximated with
  respect to the graph norms of $K$ and $\abs{\Lambda+K}^{1/2}$, respectively. For this it suffices to establish that $K$ and
  $\abs{\Lambda+K}^{1/2}$ are operator bounded with respect to $\abs{\Lambda}^{1/2}$.

  That $K$ is operator bounded with respect to $\abs{\Lambda}^{1/2}$ follows immediately from the inclusion
  $\Dom(\abs{\Lambda}^{1/2}) \subset \Dom(K)$ and the fact that $K$ is closable, see,
  e.g.,~\cite[Remark~IV.1.5 and Section~V.3.3]{Kato95}. Moreover, the same properties yield by Corollary~2.1.20 in~\cite{Tre08}
  that $K$ is operator infinitesimal with respect to $\Lambda$. In particular, the operator sum $\Lambda+K$ is self-adjoint on
  $\Dom(\Lambda+K) = \Dom(\Lambda)$ by the well-known Kato-Rellich theorem. In turn, Corollary~\ref{cor:SchmDiss} implies that
  $\Dom(\abs{\Lambda+K}^{1/2}) = \Dom(\abs{\Lambda}^{1/2})$, so that also $\abs{\Lambda+K}^{1/2}$ is operator bounded with
  respect to $\abs{\Lambda}^{1/2}$. This completes the proof of (a).

  (b).
  Corollary~\ref{cor:SchmDiss} implies that $\Dom(\abs{\Lambda+K}^{1/2}) = \Dom(\abs{\Lambda}^{1/2})$ and
  $\Dom(\abs{\Lambda}^{1/2}) \subset \Dom(\abs{K}^{1/2})$. In turn, as in part~(a), both $\abs{\Lambda+K}^{1/2}$ and
  $\abs{K}^{1/2}$ are relatively bounded with respect to $\abs{\Lambda}^{1/2}$. The claimed identity now follows just as in
  part~(a) by approximation upon observing that it certainly holds for $x,y \in \Dom(\Lambda)$.
\end{proof}%

\section*{Acknowledgements}
The author is grateful to Ivan Veseli\'c, Matthias T\"aufer, Stephan Schmitz, and Ivica Naki\'c for fruitful and inspiring
discussions. He is especially indebted to Stephan Schmitz for also commenting on an earlier version of this manuscript. Finally,
he thanks the anonymous referees of an earlier version of the manuscript for comments that helped to improve the presentation of
the material.


\end{document}